\documentclass[10pt,a4paper]{amsart}
\usepackage{amsfonts}
\usepackage{mathtools}
\usepackage{amsthm}
\usepackage{amsmath}
\usepackage{amssymb}
\usepackage{amscd}
\usepackage{cite}
\usepackage[ngerman,english]{babel}
\usepackage[mathscr]{eucal}
\usepackage{indentfirst}
\usepackage{graphicx}
\usepackage{graphics}
\usepackage{pict2e}
\numberwithin{equation}{section}
\usepackage[margin=3.5cm, bottom = 2cm, left=3.6cm, right= 4.1cm]{geometry}
\usepackage{epstopdf} 
\usepackage{bbm}
\usepackage{xcolor}
\usepackage[hidelinks]{hyperref}

\usepackage{newtxtext}       %
\usepackage{comment}

\let\OLDthebibliography\thebibliography
\renewcommand\thebibliography[1]{
	\OLDthebibliography{#1}
	\setlength{\parskip}{0pt}
	\setlength{\itemsep}{4pt plus 0.3ex}
}

\theoremstyle{plain}
\newtheorem{Th}{Theorem}[section]
\newtheorem{Lemma}[Th]{Lemma}
\newtheorem{Cor}[Th]{Corollary}

\DeclareMathOperator{\R}{\mathbb{R}}
\DeclareMathOperator{\Z}{\mathbb{Z}}

\DeclareMathOperator{\N}{\mathbb{N}}
\DeclareMathOperator{\di}{\text{div}}
\DeclareMathOperator{\supp}{\text{supp}}

\DeclareMathOperator{\dist}{\text{dist}}

\newcommand{\norm}[2]{\left\lVert #1 \right\rVert_{#2}}

\newcommand{\f}[2]{\frac{#1}{#2}}

\newcommand{\Ft}{\mathcal{F}}
\newcommand{\Sw}[1]{\mathcal{S}(\R^{#1})}

\theoremstyle{definition}
\newtheorem{Def}[Th]{Definition}

\newtheorem{Corollary}[Th]{Corollary}
\newtheorem{Rem}[Th]{Remark}
\newtheorem{?}[Th]{Problem}

\newtheorem{Propo}[Th]{Proposition}

\newcommand{\absatz}{~\\[.5cm]}

\begin{document} 
	
	\subjclass[2010]{Primary: 35A01. Secondary: 35G50}
	\keywords{biharmonic wave map, biharmonic map, fourth order, wave equation}


	\title[Global results for biharmonic wave maps]{Global results for a Cauchy problem\\ related to biharmonic wave maps}

	\author{Tobias Schmid}

	\address{EPFL SB MATH PDE,
		Bâtiment MA,
		Station 8,
		CH-1015 Lausanne}
	\email{tobias.schmid@epfl.ch}


	\begin{abstract}
		We prove global existence of a derivative bi-harmonic wave equation with a non-generic quadratic nonlinearity and small initial data in the scaling critical space $$\dot{B}^{2,1}_{\f{d}{2}}(\R^d) \times \dot{B}^{2,1}_{\f{d}{2}-2}(\R^d)$$ for $ d \geq 3 $. Since the solution persists higher regularity of the initial data, we obtain a small data global regularity result for the biharmonic wave maps equation for a certain class of target manifolds including the sphere. \end{abstract}
	
	\maketitle

	\section{Introduction}\label{sec:intro}
	In the following we consider critical points of an \emph{extrinsic (rigid)} action functional 
	\begin{align}\label{action}
		\Phi(u) = \f{1}{2}\int_{\R}\int_{\R^d} |\partial_t u |^2 - |\Delta u |^2 ~dx~dt, 
	\end{align}
	for smooth maps $ u : \R \times \R^d \to \mathbb{S}^{L-1}$  into the  round sphere $ \mathbb{S}^{L-1} \subset \R^L $ .
	Taking  smooth variations $ u_{\delta}: \R\times \R^d \to \mathbb{S}^{L-1}  $ with $ u_{\delta} - u$ having compact support and vanishing at $ \delta = 0$, the critcal points satisfy
	\begin{align}\label{ELc}
		\partial_{t}^2 u + \Delta^2 u \perp T_{u}\mathbb{S}^{L-1}  
	\end{align}
	pointwise on $ \R \times \R^d$. The variation of  \eqref{action} thus gives rise to Hamiltonian equations with \emph{elastic} energy functional
	\begin{align}\label{energy}
		E(u(t)) &= \f{1}{2}\int_{\R^d} | \partial_t u(t) |^2 + | \Delta u(t)|^2 ~dx.
	\end{align} 
	Evaluating the Euler-Lagrange equation \eqref{ELc}, we infer that critical maps $u$ of \eqref{action} are solutions of the following  \emph{biharmonic} wave maps equation
	\begin{align}\label{general}
		\partial_{t}^2u + \Delta^2u &= - |\partial_tu|^2 u - \Delta ( |\nabla u|^2) u\\[3pt] \nonumber 
		&~~~~~ - (\nabla \cdot \langle \Delta u,  \nabla u \rangle) u - \langle \nabla \Delta u, \nabla u\rangle  u\\[3pt] \nonumber
		&= - ( |\partial_t u|^2 + |\Delta u|^2  + 4\langle \nabla u, \nabla \Delta u \rangle  + 2\langle \nabla^2 u , \nabla^2 u \rangle )u,
	\end{align}
	where $ \Delta^2 $ denotes the bi-Laplacien 
	$ \Delta^2 = \Delta ( \Delta \cdot) = \partial_{i j} \partial^{i j} $  and $ \langle \nabla^2 u , \nabla^2 u \rangle = \langle \partial_i \partial^j u, \partial_j \partial^i u\rangle $.
	Hence \eqref{general} is considered to be a fourth order analogue of the (spherical) wave maps equation
	$$ -\partial_{t}^2u + \Delta u =  ( |\partial_t u |^2 - |\nabla u|^2) u,$$
	which has been studied intensively in the past concerning wellposedness,  regularity and gauge invariance, see e.g. the surveys \cite{tataru3},~\cite{Koch}. For the general wave maps equation of the form
	\begin{equation}\label{wave_map}
		\square u = \Gamma(u)( \partial_{\alpha}u, \partial^{\alpha} u) ) = \tilde{\Gamma}(u)( \square(u \cdot u) - 2 u \cdot \square u)
	\end{equation}
	local wellposedness holds almost optimal for (scaling) subcritical regularity  in $ H^s(\R^d) \times H^{s-1}(\R^d) $ with $ s > \f{d}{2}$. This relies  on the \emph{null condition} for \eqref{wave_map} as seen in the  proof of  Klainerman-Machedon in \cite{klainerman1} for $ d \geq 3 $ (and Klainerman-Selberg in \cite{klainerman2} for $d = 2$). In fact, a counterexample of Lindblad in \cite{lindblad} shows that if this condition is absent in a generic wave equation, the sharp regularity for local existence is strictly above $ d/2$.\\[6pt]
	Many advances towards (critical) regularity $s = d/2$ lead to insights for the wave maps equation with impact on related equations. For instance, global solutions with small initial data in the space $ \dot{H}^{\f{d}{2}} \times \dot{H}^{\f{d}{2}-1} $ were constructed by Tao in \cite{tao1} ($ d \geq 5 $) for sphere targets using a novel microlocal renormalization procedure. 
	The $(2+1)$ dimensional case, i.e. small data in the energy space $ \dot{H}^{1}(\R^2) \times L^2(\R^2) $, was for example treated by Krieger in \cite{krieger} with $\mathbb{H}^2$ target space, Tao in \cite{tao2} for the sphere target and Tataru \cite{tataru4} for more general targets. Global wave maps with large initial energy were considered by Krieger-Schlag in \cite{krieger-schlag} (for the $\mathbb{H}^2$ target) and in the  analysis of Sterbenz-Tataru in \cite{Ster-Tat1}, \cite{Ster-Tat2}.
	Since the  literature is vast and the list is not exhaustive, we refer e.g. to \cite{geba-grillakis} for a general overview. \absatz
	In this article, we study the analogue of the \emph{division problem} for wave maps with small data in  $ \dot{B}^{2,1}_{\f{d}{2}}(\R^d) \times \dot{B}^{2,1}_{\f{d}{2}-1}(\R^d)$, which has been solved by Tataru in \cite{tataru1} (for $d \geq 4 $) and in low dimension \cite{tataru2} (i.e. for $ d = 2,3$) by the use of  \emph{null-frame estimates}. More  recently, the division problem for wave maps (in dimension $ d \geq 2 $) has also been solved in a $ U^2 $ based space by Candy-Herr in \cite{candy-herr}. For \eqref{general}, we achieve to solve the division problem in dimension $ d \geq 3 $ using spaces $Z,~W = L(Z)$ which are the analogues of Tataru's $ F , ~\square F $ spaces in \cite{tataru1}. Especially
	$L : Z \to W$
	is a continuous operator. The results in this article are part of the authors PhD thesis \cite{Schmid}.
	\subsection{The Cauchy problem and outline}
	We consider the following generalized Cauchy problem
	\begin{align}\label{general2}
		\begin{cases}
			\partial_{t}^2u + \Delta^2u = Q_u(u_t,u_t) + Q_u(\Delta u,\Delta u) + 2 Q_u( \nabla u, \nabla \Delta u) &\\[4pt]
			~~~~~~~~~~~~~~~~~~~~~~~~~~+ 2 Q_u( \nabla \Delta u, \nabla  u) + 2 Q_u( \nabla^2 u, \nabla^2 u)  =: \mathcal{Q}(u), &\\[4pt]
			(u(0), \partial_t u(0)) = (u_0, u_1)&
		\end{cases}
	\end{align}
	where 
	\begin{align*}
		\mathcal{Q}^J(u) &= [Q_u]^J_{K,M}(\partial_t u^K \partial_tu^M) + [Q_u]^J_{K,M}(\Delta u^K\Delta u^M) + 2[Q_u]^J_{K,M}( \partial_i u^K  \partial^i \Delta u^M)\\[3pt] \nonumber
		&~~~+ 2[Q_u]^J_{K,M}( \partial_i \Delta  u^K  \partial^i  u^M) +  2 [Q_u]^J_{K,M}( \partial_i \partial^j u^K  \partial_j \partial^i u^M),
	\end{align*}
	and $ \{ Q_x~|~ x \in \R^L\} $ is  a smooth family of bilinear forms (in fact required to be analytic at the origin). Here we contract the derivatives over $ i = 1, \dots, d$ and the components of $u$ over $ K,M, J \in \{ 1, \dots, L \} $. 
	The bilinear term $ \mathcal{Q}(u) $ in \eqref{general2} is \emph{non-generic} for our results, in the sense that for bilinear interactions, the set of \emph{resonances}
	$$ \big \{ ( (\tau_1, \xi_1), (\tau_2, \xi_2))~|~ (\tau_1 + \tau_2)^2 - |\xi_1 + \xi_2|^4 = \tau_1^2 + \tau_2^2 - |\xi_1|^4 - |\xi_2|^4  \big \},$$
	is canceled by $ \mathcal{Q}(u)$. We use this fact in the form of the following commutator identity for the operator $ L = \partial_{t}^2 + \Delta^2 $
	\begin{align}\label{null-structure}
		\mathcal{Q}(u) &= \f{1}{2}Q_u (L(u \cdot u) -  u \cdot L u -  Lu \cdot u)\\[3pt] \nonumber
		&= \f{1}{2}[Q_u]_{K, M} (L(u^K \cdot u^M) -  u^K \cdot L u^M -  u^M \cdot Lu^K).
	\end{align}
	This will then be exploited following the work of Tataru in \cite{tataru1}, \cite{tataru2} for wave maps. To be precise, the idea used in Tataru's $ F,~ \square F$ spaces from \cite{tataru1} allow to treat $ \mathcal{Q}(u) $ by continuity of $L$. As a consequence, we find a simple way to solve the disvision problem for \eqref{general2} even in \emph{low dimensions} compared to the energy scaling  (of \eqref{energy}) for biharmonic wave maps \eqref{general}, see e.g. the remark \ref{remmy} below. However, we do not obtain scattering at $t \to \pm \infty$ from this approach.\\[3pt]
	The \emph{main difference} to \cite{tataru1} is that we have to use the control of a lateral Strichartz space and a maximal function bound in order to exploit a smoothing effect for the Schr\"odinger group. More details are given below.
	\absatz 
	The following second Cauchy problem will be solved with the same approach (presented in the following Sections) and further (in Section \ref{sec:proof}) applies to solve \eqref{ELc} for more general target manifolds.\\[3pt]
	The general biharmonic wave maps equation for maps $ u : [0,T) \times \R^d \to N$ into a general embedded manifold $N \subset \R^L $ reads similarly as to the spherical case
	$$ L u \perp T_u N~~~\text{on}~~~(0,T) \times \R^d. $$
	Via the smooth family of orthogonal tangent projector $ P_u : \R^L \to T_uN $, this is equivalent to
	\begin{align}\label{jap}
		Lu = ( I - P_u)(Lu) = ( I - d\Pi_u)(Lu),
	\end{align}
	where $ \Pi$ is nearest point projector $ \Pi : \mathcal{V}_{\varepsilon}(N) \to N,~~~ | \Pi(p) - p| = \inf_{q \in N} | q-p|$ with the identity $ d\Pi_u = P_u $ in case $ u $ maps to $N$ and $ \mathcal{V}_{\varepsilon}(N) = \{ p ~|~ \dist(p, N ) < \varepsilon\} $. Using $ (\partial_t, \nabla u)\in T_uN $ \eqref{jap} is expanded into a semilinear equation, where in contrast to \eqref{general}, tri-linear and quadri-linear terms appear on the RHS. For the sake of readability, we give the expansion of \eqref{jap} in Section \ref{sec:proof}.
	We hence consider the Cauchy problem
	\begin{align}\label{bihom}
		\begin{cases}
			Lu = L(\Pi(u))- d\Pi_u( Lu) &\\[6pt]
			(u(0), \partial_t u(0)) = (u_0, u_1),&
		\end{cases}
	\end{align}
	where $ \Pi : \R^L \to \R^L $ is smooth and real analytic at $ x_0 = 0$. Calculating the series expansion of $\Pi$ in the RHS of \eqref{bihom}, we infer formally
	$$ L(\Pi(u))- d\Pi_u( Lu)  = \sum_{ k \geq 2 } C_k d^k\Pi_0(L(u^k) - k u^{k-1} Lu).$$
	Thus the nonlinearity on the RHS reduces the same non-resonant form \eqref{null-structure}, where it is later justified, by the spaces we use, to commute $L$ with the series expansion. The  resulting theorem is given below in Corollary  \ref{main2}.
	\absatz
	At least formally, Duhamel's formula is given by
	\begin{align}\label{Duhamel}
		\begin{pmatrix}
			u(t)\\
			u_t(t)
		\end{pmatrix} &= S(t) \cdot \begin{pmatrix}
			u_0\\
			u_1
		\end{pmatrix}
		+
		\int_0^t S(t-s) \cdot \begin{pmatrix}
			0\\Lu(s)
		\end{pmatrix}~ds,
	\end{align}
	where
	\begin{align*}
		S(t) = \begin{pmatrix}
			\cos( (-\Delta)t) & (-\Delta)^{-1} \sin( (- \Delta) t)\\[2pt]
			\Delta \sin((-\Delta)t) & \cos( (-\Delta)t) 
		\end{pmatrix} = \f{1}{2} Q^{-1} \begin{pmatrix}
			e^{-it \Delta} & 0\\
			0 & e^{it \Delta}
		\end{pmatrix} Q
	\end{align*}
	with
	\begin{align}
		Q = \begin{pmatrix}
			-i \Delta & 1\\
			i \Delta & 1
		\end{pmatrix}.
	\end{align}
	Thus, in the analysis for \emph{biharmonic wave maps} \eqref{general}, it is in principle possible to exploit methods developed for derivative Schr\"odinger equations, which will become apparent below.\\[4pt]
	Results on the division problem for Schr\"odinger maps, see e.g. \cite{bejenaru1}, \cite{ionescu-kenig2},  involve versions of lateral Strichartz estimates in the norm ($ x \mapsto x_e e  + x_{ e^{\perp}},~e \in \mathbb{S}^{d-1}$)
	\begin{align*}
		\norm{f}{L^p_e L^q_{t, e^{\perp}}}^p =  \int_{-\infty}^{\infty} \left( \int_{[e]^{\perp}} \int_{- \infty}^{\infty}  | f( t, r e + x)|^q d t~d x\right)^{\f{p}{q}}d r,
	\end{align*}
	in order to exploit smoothing effects for Schr\"odinger equations, see the Appendix \ref{appendix}. Especially, we likewise rely on factoring 
	$$   L^{\infty}_e L^2_{t, e^{\perp}} \cdot L^{2}_e L^{\infty}_{t, e^{\perp}} \subset L^2_{t,x},  $$
	where the (lateral) energy $ L^{\infty}_e L^2_{t, e^{\perp}}$ gives additional regularity of order $ |\nabla|^{\f12}$ and the maximal function bound $  L^{2}_e L^{\infty}_{t, e^{\perp}} $ is controlled uniform in  $ e \in \mathbb{S}^{d-1}$. Apart from the usual Strichartz space $S_{\lambda}$, this will be essential (in one particular frequency interaction) in Section \ref{sec:multi}.
	\absatz
	The operator $ Lu = \partial_{t}^2u + \Delta^2u$ appears in the \emph{Euler-Bernoulli beam model} ($ d = 1$) and in effective \emph{thin-plate equations} ($d=2$) such as the Kirchoff- and Von K\'arm\'an elastic plate models with small plate deflections if rotational forces are neglected.  As a reference we mention e.g. \cite{plate}, where this situation has been considered explicitly with (nonlinear) boundary dissipation.

	~~\\
	\begin{center}
		\large \emph{Outline of the article}
	\end{center}
	~~\\
	In Section \ref{subsec:Linear-est}, we provide (lateral) Strichartz estimates $L^{p}_{e} L^q_{t, e^{\perp}}$ and the $ L^{2}_{e} L^{\infty}_{t, e^{\perp}}$ estimate for the linear Cauchy problem of the operator $L = \partial_{t}^2 + \Delta^2 $. This is a consequence of the corresponding estimates for $ e^{\pm it \Delta }$  which orginally appeared in \cite{ionescu-kenig1}, \cite{ionescu-kenig2} and \cite{bejenaru1}. In the Appendix \ref{appendix}, we briefly outline proofs of the Strichartz estimates we need for $ e^{\pm i t \Delta}$ based on the calculation by Bejenaru in \cite{bejenaru1}.
	\absatz
	In Section \ref{subsec:function-spaces}, we construct spaces $Z^{\f{d}{2}}, W^{\f{d}{2}}$ such that
	\begin{align}
		&Z^{\f{d}{2}} \subset C(\R,\dot{B}^{2,1}_{\f{d}{2}}(\R^d))\cap \dot{C}^1(\R,\dot{B}^{2,1}_{\f{d}{2}-2}(\R^d)),\\[3pt]~~~ &\norm{u}{Z^{\f{d}{2}}} \lesssim \norm{(u_0, u_1)}{\dot{B}^{2,1}_{\f{d}{2}} \times  \dot{B}^{2,1}_{\f{d}{2}-2}} + \norm{Lu}{W^{\f{d}{2}}},
	\end{align}
	and similar $ Z^{s}, W^{s}$ for $ s > \f{d}{2} $ with data in $ \dot{H}^s(\R^d) \times \dot{H}^{s-2}(\R^d)$.\\[4pt] Further, we prove the algebra properties
	\begin{align}\label{bilinear-embeddings1}
		&Z^{\f{d}{2}} \cdot Z^{\f{d}{2}}  \subset Z^{\f{d}{2}},\\[3pt]
		&W^{\f{d}{2}} \cdot Z^{\f{d}{2}}  \subset W^{\f{d}{2}},\label{bilinear-embeddings2}
	\end{align}
	in Section \ref{sec:multi}.  For the higher regularity, we need to provide the following embeddings
	\begin{align}
		&(Z^{\f{d}{2}} \cap Z^s) \cdot (Z^{\f{d}{2}} \cap Z^s)  \subset Z^{\f{d}{2}}\cap Z^s,\label{klar1}\\[3pt]
		&(W^{\f{d}{2}} \cap W^s) \cdot (Z^{\f{d}{2}}\cap Z^s)  \subset W^{\f{d}{2}} \cap W^s.\label{klar2}
	\end{align}
	To be more precise it suffices, as in \cite{tataru1} and \cite{bejenaru1}, to conclude from the dyadic estimates 
	\begin{align}
		\label{higher-regularity1}
		&\norm{uv}{Z^s} \lesssim \norm{u}{Z^s} \norm{v}{Z^{\f{d}{2}}} + \norm{v}{Z^s} \norm{u}{Z^{\f{d}{2}}},~~ u , v \in Z^{\f{d}{2}} \cap Z^s,\\[3pt]
		&\norm{uv}{W^s} \lesssim \norm{u}{W^s} \norm{v}{Z^{\f{d}{2}}} + \norm{v}{Z^s} \norm{u}{W^{\f{d}{2}}},~~ u \in W^{\f{d}{2}} \cap W^s ,~ v \in Z^{\f{d}{2}} \cap Z^s,\label{higher-regularity2}
	\end{align}
	which is outlined in  Section \ref{sec:multi} for $ s > \f{d}{2}$.
	Finally, we sketch the fixed point argument from \cite{tataru1} and the application to biharmonic wave maps stated in Corollary \ref{main2} in Section \ref{sec:proof}.
	\absatz
	We emphasize that the construction of the dyadic blocks $ Z_{\lambda}, W_{\lambda} $ are the analogues of Tataru's $F_{\lambda} ,~\square F_{\lambda} $ spaces in \cite{tataru1}, since we globally bound $ L u $ in the spaces $ L_t^1 L^2_x$. In particular, the operator
	$$ L : Z_{\lambda}  \to W_{\lambda} $$
	is  \emph{continuous by construction} of $Z_{\lambda}$ and $W_{\lambda}$. Combining this with \eqref{bilinear-embeddings1} and \eqref{bilinear-embeddings2}, it suffices to estimate  $\mathcal{Q}(u) $ in \eqref{general2} with the identity \eqref{null-structure}.\\[3pt]
	As mentioned above, we can not fully rely on the usual Strichartz norm and  have to use the control of the \emph{lateral} Strichartz norm, which exploits additional smoothing in the proof of \eqref{bilinear-embeddings1}. This idea has been used in the similar context of the Schr\"odinger maps flow by Ionescu-Kenig \cite{ionescu-kenig1}, \cite{ionescu-kenig2}, Bejenaru \cite{bejenaru1} and Bejenaru-Ionescu-Kenig \cite{bejenaru2}.

	\subsection{The main results}\label{main-results}
	
	The system \eqref{general2} is largely motivated by biharmonic wave maps, however the results for
	\eqref{general2} are based on the structural extension of evolution equtions with a nonlinearity that, due to \eqref{null-structure}, can be considered \emph{non-generic}.\\[4pt]
	We turn to general systems  \eqref{general2} and \eqref{bihom} for functions $u^1,~\dots,~ u^{L} $ with $L \in \N$, where we assume that $ x \mapsto Q_x $, $ x \mapsto \Pi(x)$ are \emph{real analytic} in the point $x_0 = 0$.
	
	\begin{Th}\label{main1} \begin{itemize} 
			\item[(i)] For $d \geq 3 $ there exists $ \delta  > 0 $ sufficiently small such that the following holds. Let $ (u_0, u_1) \in \dot{B}_{\f{d}{2}}^{2,1} (\R^d)\times \dot{B}_{\f{nd}{2}-2}^{2,1}(\R^d) $ with
			\begin{align}\label{small}
				\norm{u_0 }{ \dot{B}_{\f{d}{2}}^{2,1}(\R^d)} + \norm{u_1}{\dot{B}_{\f{d}{2}-2}^{2,1}(\R^d)} \leq  \delta. 
			\end{align}	
			Then \eqref{general2} and \eqref{bihom} have a global solution $ u\in C(\R, \dot{B}^{2,1}_{\f{d}{2}}(\R^d))\cap \dot{C}^1(\R, \dot{B}^{2,1}_{\f{d}{2}-2}(\R^d)) $  with 
			\begin{align}
				\sup_{t \geq 0 } 
				\big( \norm{ u(t) }{\dot{B}_{\f{d}{2}}^{2,1}(\R^d) } + \norm{\partial_tu(t)}{\dot{B}_{\f{d}{2}-2}^{2,1}(\R^d)} \big) \leq C \delta,
			\end{align}
			for some $ C> 0 $. Further, the solution depends Lipschitz on the initial data.\\
			\item[(ii)] If additionally $(u_0 , u_1) \in \dot{H}^{s }(\R^d) \times  \dot{H}^{s-2}(\R^d) $ for some $ s > \f{d}{2}$, then also $ (u(t) , \partial_t u(t)) \in \dot{H}^{s }(\R^d) \times  \dot{H}^{s-2}(\R^d)$ for all $ t \in \R$ and in fact
			$$ \sup_{t \geq 0 } \big( \norm{ u(t) }{\dot{H}^{s}(\R^d) } + \norm{\partial_tu(t)}{\dot{H}^{s-2}(\R^d)} \big) \leq C (\norm{u_0}{\dot{H}^{s }(\R^d)} + \norm{u_1}{\dot{H}^{s-2}(\R^d)}).$$
		\end{itemize}
	\end{Th}

	~~\\
	This theorem applies to \eqref{general}, however it is not clear if the solution maps to $ \mathbb{S}^{L-1}$ for all times. This is proven within the following (slightly more general) setup. Let $ N \subset \R^L $ be an embedded manifold and such that the nearest point projector $ \Pi : \mathcal{V}_{\varepsilon}(N) \to N $ is analytic on $N$ with a uniform lower bound on the radius of convergence. An explicit example is a uniformly analytic pertubation of the round sphere $ \mathbb{S}^{L-1}$.  
	\begin{Cor} \label{main2}  Let $(u_0, u_1) : \R^d \to TN$, i.e. $ u_0 \in N,~~ u_1 \in T_{u_0}N $, be a smooth map such that $ \supp(\nabla u_0, u_1 ) $ is compact,~$ d \geq 3$. Then if 
		\[ 	\norm{u_0  }{ \dot{B}_{\f{d}{2}}^{2,1}(\R^d)} + \norm{u_1}{\dot{B}_{\f{d}{2}-2}^{2,1}(\R^d)} \leq  \delta,\]
		where $ \delta = \delta(d, N) > 0 $ is sufficienty small, then \eqref{jap}, i.e.
		$$ \partial_t^2 u + \Delta^2 u \perp T_uN $$
		has a global smooth solution $ u : \R \times \R^d \to N  $ with $ (u(0), \partial_tu(0)) = (u_0, u_1)$.
	\end{Cor}
	\begin{Rem}
		The statement of Corollary \ref{main2} has to be rigorously corrected to $ u - p \in \dot{B}_{\f{d}{2}}^{2,1}(\R^d) $ for $ p = \lim_{ x \to \infty} u_0(x) $ since $ u_0 : \R^d \to N $ has no decay.
	\end{Rem}
	~~\\
	In \cite{HLSS} the authors proved local wellposedness of the Cauchy problem for \eqref{ELc}, resp. the expansion of \eqref{jap} for general compact target manifolds $N$, in the Sobolev space $H^k \times H^{k-2} $ for $k \in \Z$ with $ k > \lfloor \f{d}{2} \rfloor + 2 $. Here, the term involving $ \nabla^3 u $ in \eqref{general} did not allow for a direct energy estimate, but instead required a parabolic regularization similar as to the classical work of e.g. Bona-Smith.\\
	This approach, however, uses priori energy estimates that rely on the geometric condition \eqref{ELc}.
	In a recent preprint the author proved in a similar mannar via energy estimates that if $(u(0), \partial_tu(0))$ are smooth with compact support (ie. $u(0)$ is constant outside of a compact subset of $ \R^d$) and $ d \in \{ 1,2\}$, then \eqref{jap} has a global smooth solution.\\
	We further mention that Herr, Lamm and Schnaubelt proved the existence of a global weak solution, see \cite{HLS}, for the case $ N = \mathbb{S}^{L-1}$ by a Ginzburg-Landau approximation and the use of Noether's law for the sphere.
	
	\begin{Rem}\label{remmy} The equation \eqref{general} has parabolic scaling
		$$ u_{\lambda}(t, x) = u(\lambda^2 t, \lambda x),~ x \in \R^d, t \in \R. $$
		Thus it holds 
		$$ \lambda^{4-d}E(u(\lambda^2 t)) = E(u_{\lambda}(t))$$
		and   $ d = 4 $, the  (energy) critical dimension, is included in our results Theorem \ref{main1} and Corollary \ref{main2}. This is due to the larger Strichartz range for the dispersion rate $d/2$, whereas the low dimensional case for wave maps is more involved than \cite{tataru1} and has first been solved by Tataru in \cite{tataru2}.
	\end{Rem}
	\section{Linear estimates and  function spaces}\label{sec:Linear-est-func}
	\subsection{Notation}
	For real $A,B \geq 0 $ we write $ A \lesssim B$ short for  $ A \leq c B $, where $ c > 0 $ is a constant. Likewise we write $ A \sim B $ if there holds $ A \lesssim B $ and $ B \lesssim A$. The space of Schwartz functions will be denoted by $ \mathcal{S}$ and 
	the Fourier transform for $ u \in \Sw{d}$ will be  
	\begin{align}
		\Ft(u)(\xi)  = \int_{\R^d} e^{- i  x \cdot \xi } u(x)~dx,
	\end{align}
	for which we write $\hat{u}(\xi) = \Ft(u)(\xi)$. We indicate by $ \mathcal{F}_{x'}(\xi) $ that the Fourier transform is taken over $ x'$ where $x = (x', \tilde{x})$ if necessary and let $ \varphi \in C^{\infty}(\R) $ be a Littlewood-Paley function, i.e. such that
	\begin{align}
		& \supp(\varphi) \subset  ( \f{1}{2}, 2  ),~~\varphi \in [0,1],~~~ \text{and}~~\sum_{j \in \Z} \varphi( 2^{-j} s) = 1,~~\text{for}~ s > 0.
	\end{align}
	We define the multiplier  $P,Q$ for $ u \in \mathcal{S}'(\R^d),~ v \in \mathcal{S}'(\R^{1+d})$  and dyadic numbers $ \lambda, \mu$ by
	\begin{align*}
		&\widehat{P_{\lambda}( \nabla)u}(\xi) = \varphi(| \xi|/\lambda )\hat{u}(\xi),~~~\widehat{P_{\lambda}(D)v}(\tau, \xi) = \varphi(( \tau^2 + |\xi|^4)^{\f{1}{4}}/ \lambda )\hat{v}(\tau, \xi),\\[8pt]
		&\widehat{Q_{\mu}(D)v}(\tau, \xi) = \varphi(w(\tau, \xi)/ \mu  )\hat{v}(\tau, \xi),\\[8pt]
		&P_{\leq \lambda}= \sum_{ \tilde{\lambda}\leq \lambda } P_{\tilde{\lambda}},~~Q_{\leq \mu}= \sum_{ \tilde{\mu}\leq \mu } Q_{\tilde{\mu}},
	\end{align*}
	where 
	$$ w(\tau, \xi)  = \frac{|\tau^2 - |\xi|^4|}{(\tau^2 + |\xi|^4)^{\f12}} \sim ||\tau| - \xi^2|,~~~~~(\tau^2 + |\xi|^4)^{\f14} \sim (|\tau| + \xi^2)^{\f12}.$$
	Further, we write $ v_{\lambda} = P_{\lambda}v = P_{\lambda}(D)v  ,~ P_{\lambda, \leq \mu}  = P_{\lambda} Q_{\leq \mu}(D) $ for short and define
	\begin{align*}
		A_{\lambda} = \{ (\tau, \xi)~|~ \lambda \slash 2 \leq (\tau^2 + \xi^4 )^{\f14} \leq 2 \lambda \},~~~~A_{\lambda}^d = \{  \xi~|~ \lambda \slash 2 \leq  |\xi| \leq 2 \lambda \}.
	\end{align*}
	For a distribution $ f \in \mathcal{S}'(\R^{d+1})$ we say  $ f$ is \emph{localized at frequency} $\lambda \in 2^{\Z} $ if 
	$ \hat{f} $ has support in the set $  A_{\lambda}$
	and a similar notation is used for  $ g \in \mathcal{S}'(\R^d)$ and $  A_{\lambda}^d$. 
	In addition, we need to localize in the sets
	\[   A_e : = \left \{ \xi~|~ \xi \cdot e \geq \f{|\xi|}{\sqrt{2}} \right \},~~ e \in \mathbb{S}^{d-1}, \]
	in order to exploit the smoothing effect for the linear equation. Thus, as in \cite{bejenaru1}, we choose $ \mathcal{M} \subset \mathbb{S}^{d-1}$ with $ e \in \mathcal{M} \Rightarrow - e \in \mathcal{M}$ such that
	\begin{align}
		& \R^d = \bigcup_{ e \in \mathcal{M}} A_e,~~\forall e \in \mathcal{M}~:~\# \{ \tilde{e} \in \mathcal{M}~|~ A_e \cap A_{\tilde{e}} \neq \emptyset~\} \leq K,
	\end{align}
	with a constant $K = K_d > 0 $. Further we require a smooth partition of unity $ \{ h_e\}_{ e \in \mathcal{M}}$  subordinate to $\{ A_e\}_{ e \in \mathcal{M}}$, i.e. 
	\begin{align}
		&h_e \in C^{\infty}(\R^d),~~\supp(h_e ) \subset A_e,~~ h_e \in [0,1]\\[2pt]
		&\sum_{e \in \mathcal{M}} h_e(\xi ) = 1,~~ \xi \in \R^d \backslash \{0\}.
	\end{align}
	We note that this is possible since in particular for $ x \in \R^d \backslash \{ 0 \} $ we have  $ x \in A_e $ if and only if $  \measuredangle(x, e) \leq \f{\pi}{4} $.
	We define the respective Fourier multiplier by
	\begin{align} 
		\widehat{P_e(\nabla) v}(\tau, \xi) = h_e(\xi ) \hat{v}(\tau, \xi),~~ v \in \mathcal{S}'(\R^{d+1}).
	\end{align}
	Finally, we choose  $\chi \in C^{\infty}(\R^{d+1}) $ such that
	\begin{align*}\chi(\tau, \xi) =
		\begin{cases}
			1 &~~| \tau^2 - |\xi|^4 | < \frac{\tau^2 + |\xi|^4}{100},\\[3pt]
			0 &~~| \tau^2 - |\xi|^4 | > \frac{\tau^2 + |\xi|^4}{10}.
		\end{cases}
	\end{align*}
	In order to have $\chi$ invariant under parabolic scaling, we choose 
	$ \chi(\tau, \xi) = \eta ( |\f{\tau^2 - |\xi|^4}{\tau^2 + |\xi|^4} |),$
	where $ \eta \in C^{\infty}(\R) $ with $ 0 \leq \eta \leq 1 $ and such hat 
	$\eta(x) = 1 $ if $ |x| < 1/100$ and $ \eta(x) = 0 $ if $ |x| > 1/10$.
	We then define
	\begin{align}
		&\widehat{P_0v}(\tau, \xi) = \chi(\tau,  \xi)\hat{v}(\tau, \xi),~~~\widehat{(1 - P_0)v}(\tau, \xi) = ( 1 - \chi(\tau,  \xi))\hat{v}(\tau, \xi).
	\end{align}
	Thus, we have 
	\begin{align}
		&\supp(\widehat{P_0 v}) \subset   \bigg\{ (\tau, \xi)~|~ | | \tau| - \xi^2| \leq  \frac{|\tau| + \xi^2}{10} \bigg\},\\[5pt]
		&\supp(\widehat{(1 - P_0)v}) \subset \bigg\{ (\tau, \xi)~|~ | | \tau| - \xi^2| \geq   \frac{|\tau| + \xi^2}{100} \bigg\}.
	\end{align}
	Especially, measuring the distance to the characteristic surface $P$,
	$$ \dist ((\tau, \xi), P) \sim \frac{| | \tau| - \xi^2|}{(|\tau| + \xi^2)^{\f12}},~~P = \big\{(\tau, \xi)~|~ \tau^2 = \xi^4 ~\big\},$$
	we infer that  $ (1 - P_0)v$ (with $v$ being localized at frequency $\lambda$) is localized where 
	$$  \dist( (\tau, \xi), P)  \sim \lambda,$$
	such that frequency $ (\tau^2 + |\xi|^4)^{\f14} \sim \lambda$ and modulation $ ||\tau| - \xi^2| \sim \mu$  are of comparable size $ \mu \sim \lambda $. For $P_0v$ we have localization where
	$$\dist( (\tau, \xi), P) = \mathcal{O}(\lambda),$$
	with a small constant that suffices to obtain additional smoothing in the linear estimates of the following sections.
	Further, we use the homogeneous spaces  $ \dot{B}_s^{2,p}(\R^d),~ 1 \leq p < \infty$ given by the closure of
	\[ \norm{u}{\dot{B}_s^{2,p}}^p = \sum_{ \lambda \in 2^{\Z}} \lambda^{sp}\norm{P_{\lambda}u }{L^2_{x}}^p,~ u \in  \mathcal{S}(\R^d).\]
	Hence we have $ \dot{H}^s(\R^d) \sim \dot{B}_s^{2,2}(\R^d)$ for $s > \f{d}{2}$ and $\dot{B}_{\f{d}{2}}^{2,1}(\R^d) \subset L^{\infty}(\R^d) $ is a well-defined Banach space.\\[5pt] 
	Let $ \lambda$ be a fixed dyadic number. Then for $ b \leq 1 $ and $ p \in [1, \infty)$ we set
	\begin{align}\label{normm}
		\norm{f}{X_{\lambda} ^{b,p}}^p = \sum_{\mu \in 2^{\Z}}\mu^{pb} \norm{Q_{\mu}(D)f}{L^2_{t, x}}^p,
	\end{align}
	and denote by $ X_{\lambda} ^{b,p}$ the closure of the (semi-)norm in $\mathcal{S}$ restricted to functions $f$ localized at frequency $\lambda$. This definition is extended as usual to the case $  p = \infty $.  We observe that $ f \in X_{\lambda} ^{b,p} $ has the representation
	\begin{align}\label{decom}
		f = \sum_{\mu \lesssim \lambda^2} h_{ \mu} + h,
	\end{align}
	where $h$ is a solution of $Lh = 0$ (with initial data localized at frequency $ \lambda$). Thus $f$ is only well-defined up to homogeneous solutions $Lh = 0$. In the following, we will correct \eqref{normm} by $ \norm{h}{L^{\infty}_tL_x^2} + \norm{\partial_t h}{L^{\infty}_t\dot{H}_x^{-2}}$  as a limiting dyadic block ($\mu \searrow 0$),  where $ L h = 0$ and $ h(0) = f(0), ~\partial_t h(0) = \partial_t f(0)$.\\[3pt]
	More precisely, the atomic decomposition \eqref{decom} has the form
	\begin{align} \label{decomposition} 
		& h_{\mu}(t,x) =  \int_{- \infty}^{\infty} \frac{e^{it s}}{|s|^b} h^{\mu}(s, x) ~ds,\\[2pt]
		&\norm{f}{\dot{X}_{\lambda}^{b,p}}^p \sim \sum_{\mu \lesssim \lambda^2} \left( \int_{- \infty}^{\infty} \label{norm-decomposition}  \norm{h^{\mu}(s,x)}{L^2_x}^2~ds \right)^{\f{p}{2}}.
	\end{align}
	Here the $ h^{\mu}(s, \cdot) $ solves $  L h^{\mu}(s, \cdot) = 0 $ for some $L^2 \times \dot{H}^{-2}$ initial data and is localized where $ s \sim \mu$. Further \eqref{norm-decomposition} only holds up to $ \mu = 0$ as mentioned above. We infer \eqref{decomposition} and \eqref{norm-decomposition} by foliation, which also shows that the sum in \eqref{decom} is well-defined distributionally for the cases $ b < \f{1}{2} $ and $p \geq 1 $ or $ b = \f{1}{2}$ and $ p = 1$. We will use the foliation explicitly in the proof of Lemma \ref{trans-lemma}.\\[4pt]
	\subsection{Linear estimates}\label{subsec:Linear-est}
	The goal of this section is to develope estimates for the linear equation
	\begin{align}
		\begin{cases}
			~~\partial_{t}^2u(t,x) + \Delta^2  u(t,x) = F(t,x)& ~~(t,x) \in \R \times \R^d \label{equation-linear}\\[5pt]
			~~u[0] = (u(0), \partial_t u(0)) = (u_0, u_1),& ~~\text{on}~\R^d,
		\end{cases}
	\end{align}
	with data $ F, u_0, u_1$. In the following we provide \emph{lateral Strichartz estimates} and a maximal function estimate for the Cauchy problem \eqref{equation-linear} in case $ F \in L^1_t L^2_x$. The main results of this section summarize all necessary homogeneous bounds in Lemma \ref{linear-bounds} and the inhomogeneous bounds in Lemma \ref{linear-bounds-inhomogeneous}. Further we give a proof of the trace estimate in Lemma \ref{trans-lemma}.\\[10pt]
	First, we start by recalling  the classical Strichartz estimate, which follows similar as for the linear wave equation. Since we did not find it in the literature for \eqref{equation-linear}, we briefly state the estimate.
\begin{Def} We say that a pair $(p,q) $ with $ 1 \leq p, q \leq \infty$ is \emph{admissible} in dimension $d \in \N,~ d \geq 2$ if  there holds
	\begin{align}
		\f{2}{p}+\f{d}{q} \leq \f{d}{2}.
	\end{align}
	and $ (p,q) \neq (2, \infty)$ in the case of $ d = 2$.	
\end{Def}

\begin{Lemma}[Strichartz]
	Let $u$ be a weak solution of \eqref{equation-linear} for data $u_0, u_1, F$. Then there holds
	\begin{align}\label{the-classical-strichartz}
		\norm{u}{C(\R,\dot{H}^{\gamma})}  + \norm{u}{L^p_t L^q_x} \lesssim \norm{u_0}{\dot{H}^{\gamma}} + \norm{u_1}{\dot{H}^{\gamma-2}} + \norm{F}{L^{\tilde{p}'}_tL^{\tilde{q}'}_x},
	\end{align}
	where $(p,q), (\tilde{p}, \tilde{q})$ are admissible pairs with $ q, \tilde{q} < \infty $ and $ \gamma \in [0, 2]$  satisfies
	\begin{align}\label{gap}
		\f{2}{p} + \f{d}{q} = \f{d}{2} - \gamma = \f{2}{\tilde{p}'} + \f{d}{\tilde{q}'} - 4
	\end{align}
\end{Lemma}
\begin{proof}
	
	We prove the inequality for  $ P_{\lambda}(\nabla)u,~ P_{\lambda}(\nabla)F$, where $ \lambda $ is a dyadic number. Then \eqref{the-classical-strichartz} follows by the Littlewood-Paley theorem since $ q, \tilde{q} < \infty $. Further, \eqref{the-classical-strichartz} is invariant under  scaling $$ u_{\lambda}(t,x) = u(\lambda^2t, \lambda x),~ F_{\lambda} = \lambda^{4}F(\lambda^2 t, \lambda x),$$
	which follows from \eqref{gap}. Especially, since $ (P_{\lambda}u)_{\lambda^{-1}} = P_1 u_{\lambda^{-1}}$, we assume $ \lambda \sim 1$. By Duhamels formula we obtain 
	\begin{align*}
		u(t) =& \cos( -t \Delta )u_0 + \frac{\sin(- t\Delta)}{(- \Delta)}u_1 + \int_0^t \frac{\sin(-(t-s) \Delta)}{ (-\Delta)} F(s)~ds.
	\end{align*}
	Therefore, as used above already, we decompose
	$$ \sin(-t\Delta)f = \f{1}{2i}(  e^{- i t \Delta}f - e^{ i t \Delta}f),~~\cos( - t \Delta)f = \f{1}{2}(  e^{- i t \Delta}f + e^{ i t \Delta}f),$$
	and by $ \lambda \sim 1 $ this can hence be estimated via
	$$ \widehat{U_{\pm}(t)f}(\xi) = \chi\{ t \geq 0 \} e^{\mp i t \xi^2} \psi(|\xi|) \hat{f}(\xi),~~~ f \in \mathcal{S}(\R^d),$$
	where $ \psi \in C^{\infty}_c((0, \infty))$ with $\psi(x) = 1 $ for $ x \in \supp(\varphi)$ and $ \varphi $ is a Littlewood-Paley function.
	Clearly $U_{\pm}(t)$ extends to $L^1(\R^d)\cap L^2(\R^d)$ satisfying the energy bound and for the dispersive estimate, we use 
	$$ U_{\pm}(t)f = K_{\pm}(t, \cdot ) \ast_d f,~~ K_{\pm}(t, x) = \chi\{ t \geq 0 \} \f{1}{(2 \pi)^d}\int_{\R^d} e^{ix \cdot \xi \mp i t |\xi |^2} \psi(|\xi|)~d\xi.$$
	The kernel then applies to the classical theorem for the decay of the Fourier transform of surface carried measures  on $ P_{\pm} = \{(\tau, \xi)~|~\pm \tau + \xi^2 = 0 \}$. Especially all principle curvature functions on $P$ are non-vanishing ($(\tau, \xi) \neq (0,0)$) and thus
	$$ \norm{U_{\pm}(t)f}{L^{\infty}(\R^d)} \lesssim ( 1 + |t|)^{- \f{d}{2}} \norm{f}{L^1(\R^d)}.$$
	In particular this implies the homogeneous and the inhomogeneous estimate by the $TT^*$ principle and interpolation (respectively Keel-Tao's endpoint argument) . Thus \eqref{the-classical-strichartz} holds on $ [0, \infty)$ and we apply this inequality  to $ u_-(t,x) =  u(-t,x),~ F_-(t,x) = F(-t, x) $, which in turn implies the full estimate. It remains to prove  $u \in  C_tL^2_x$, which follows analogously as for the wave equation.
\end{proof}
~~\\
We note that due to $ (- \Delta)^{-1}$ in Duhamel's formula, the frequency localization in the proof is evitable in contrast to the direct decay estimate for the kernel of the Schr\"odinger group.
\begin{Corollary}\label{Strichh}
	Let $ u$ have Fourier support in $ A_{\lambda}$. Then 
	\begin{align}\label{Strichhbound}
		\norm{u}{S_{\lambda}}  \lesssim \norm{u(0)}{L^2} + \norm{\partial_t u(0)}{\dot{H}^{ -2}} + \lambda^{-2} \norm{Lu}{L^1_t L^2_x},
	\end{align}
	where
	\begin{align*}
		S_{\lambda} = \big\{ f \in C_t L^2_x~|~ \supp(\hat{f}) \subset A_{\lambda},~\norm{f}{S_{\lambda}} = \ \sup_{(p,q)} \big( \lambda^{\f{2}{p} + \f{d}{q} - \f{d}{2} }\norm{f}{L^p_tL^q_{x}} \big) < \infty ~\big\}
	\end{align*}
	and the supremum is taken over admissible pairs $(p,q)$.
\end{Corollary}
\begin{proof}
	From Bernstein's estimate
	$$ \norm{u_0}{\dot{H}^{\gamma}} + \norm{u_1}{\dot{H}^{\gamma-2}} \lesssim \lambda^{\gamma}( \norm{u_0}{L^2} + \norm{u_1}{\dot{H}^{-2}}),$$
	which by \eqref{the-classical-strichartz} and the gap \eqref{gap} implies the desired estimate for all admissible pairs $(p,q)$ with $ q < \infty $. For the case $ q = \infty$ in $ d \geq 3$  we estimate by Soblev embedding (or Bernstein's bound) for any $ q \geq \f{2d}{d-2}$
	\begin{align*}
		\lambda^{-\f{d-2}{2}} \norm{u}{L^2_t L^{\infty}_x} &\lesssim \lambda^{-\f{d}{2} + 1 + \f{d}{q}} \norm{u}{L^2_t L^q_x}\\
		&\lesssim \norm{u(0)}{L^2} + \norm{\partial_t u(0)}{\dot{H}^{ -2}} + \lambda^{-2} \norm{Lu}{L^1_t L^2_x}.
	\end{align*}
\end{proof}
The Corollary \ref{Strichh} is not sufficient for our proof of bilinear estimates in Section \ref{sec:multi} and we additionally need to apply a well known smoothing estimate for the Schr\"odinger group.

For this reason, we define the following norm (see also \ref{appendix})
\begin{align}
	\norm{u}{L^p_e L^q_{t,e^{\perp}}}^p = \int_{-\infty}^{\infty} \left( \int_{[e]^{\perp}} \int_{- \infty}^{\infty}  | u( t, r e + x)|^q d t~d x \right)^{\f{p}{q}}d r,~~ e \in \mathbb{S}^{d-1}.
\end{align}  
In order to introduce the necessary notation, we recall Bejenaru's calculus from \cite{bejenaru1} (see also the Appendix \ref{appendix}) for the Cauchy problem \eqref{equation-linear}. In the case $ F = 0$ we have 
$$ \supp(\hat{u}) \subset P = \{ (\tau, \xi)~|~ \tau^2 - |\xi|^4 = 0\}, $$
which is a paraboloid in the variables $ (\tau, \xi)$. More precisely, denoting by $\Xi = (\tau, \xi) $ the Fourier variables, we split the symbol (in case of general $F$)
\begin{align}
	\widehat{F}(\tau, \xi) = L(\Xi) \hat{u}(\tau, \xi) = -(\tau - \xi^2)(\tau + \xi^2)\hat{u}(\tau, \xi).
\end{align}  
Hence, we further split in the Fourier space into 
\begin{align}\label{split}
	&- (\tau + \xi^2)^{-1}\widehat{F}(\tau, \xi)\chi\{ \tau > 0\}  = (\tau - \xi^2)\hat{u}(\tau, \xi)\chi\{ \tau > 0\},\\ \label{ssplit}
	&( - \tau + \xi^2)^{-1}\widehat{F}(\tau, \xi)\chi\{ \tau \leq  0\}  =  ( \tau + \xi^2)\hat{u}(\tau, \xi)\chi\{ \tau \leq 0\},~~|\xi| \neq 0,
\end{align}
and introduce coordinates adapted to a characteristic unit normal $ e \in \mathbb{S}^{d-1} $. That means we use the change of coordinates
$$ \Xi \mapsto (\tau, \xi \cdot e, \xi - (\xi \cdot e) e) =: ( \tau, \xi_e, \xi_{e^{\perp}}) =: \Xi_e,$$ 
and the sets
\begin{align}
	&A_e  = \left \{ \xi~|~ \xi_e \geq \f{|\xi|}{\sqrt{2}} \right \},~~~B_e : = \left \{ (\tau,\xi)~|~ |  |\tau| - \xi^2| \leq \f{|\tau| + \xi^2}{10},~ \xi \in A_e \right \}\\[5pt]
	&B^{\pm}_e : = \left \{ (\tau,\xi)~|~ | \pm \tau - \xi^2| \leq \f{|\tau| + \xi^2}{10},~ \xi \in A_e \right \} = B_e \cap \{ \pm \tau > 0 \} \cup \{ (0,0)\}.
\end{align}
Then for any $ (\tau, \xi) \in B_e $, we clearly have
\begin{align}\label{facts}
	|\tau| - \xi_{e^{\perp}}^2 \geq 0,~~ \xi_e \sim (|\tau| + \xi^2)^{\f12},~~\xi_e + \sqrt{ |\tau| - \xi_{e^{\perp}}^2} \sim (|\tau| + \xi^2)^{\f12},
\end{align} 
and similar for $ \pm \tau$ on $  B_e^{\pm}$.\\[5pt]
Especially, the latter two quantities in \eqref{facts} are controlled by frequency. Also, if we assume that $ \supp(\hat{u}) \subset B_e$, then for $ |\tau| + \xi^2 > 0$, we have from \eqref{split} and \eqref{ssplit} 
\begin{align}\label{ultrasplit}
	- ( |\tau| + \xi^2)^{-1}\left( \xi_e + \sqrt{ |\tau| - \xi_{e^{\perp}}^2}\right)^{-1} \widehat{F}(\tau, \xi)  = \left(\sqrt{ |\tau| - \xi_{e^{\perp}}^2} - \xi_e \right) \hat{u}(\tau, \xi),
\end{align}
Now, taking the FT in the variable $ \Xi_e $, we obtain that \eqref{equation-linear} is equivalent to 
\begin{align}\label{pseudo-evo1}
	\big ( i \partial_{x_e} +  D_{t, e^{\perp}}\big ) \tilde{u}(t, x_e, x_{e^{\perp}}) = \tilde{F}(t, x_e, x_{e^{\perp}}), 
\end{align}
where
\begin{align}\label{supi-op}
	\widehat{D_{t, e^{\perp}} u}(\tau, \xi_e,\xi_{e^{\perp}} ) &= \left(\sqrt{ |\tau|- \xi_{e^{\perp}}^2}\right) \hat{u}(\tau, \xi),\\[3pt]
	\label{split3}
	\mathcal{F}(\tilde{F})(\tau, \xi_e, \xi_{e^{\perp}}) &= -(|\tau| + \xi^2)^{-1}\left( \xi_e + \sqrt{ |\tau| - \xi_{e^{\perp}}^2}\right)^{-1}\widehat{F}(\tau, e \xi_e + \xi_{e^{\perp}})\\[3pt]
	\mathcal{F}(\tilde{u})(\tau, \xi_e, \xi_{e^{\perp}}) &= \hat{u}(\tau, \xi_e e + \xi_{e^{\perp}}),
	\label{supi-u^-}
\end{align}
\begin{Rem}The calculations above apply to prove inhomogeneous linear estimates for \eqref{equation-linear} with $ F \in L^1_e L^2_{t, e^{\perp}}$ that are based on on the reduction to \eqref{pseudo-evo1}. However, using the above notation for the sets $B_e$ and $ A_e$, we only need estimates for $F \in  L^1_t L^2_x $ localized on $ B_e \cap A_{\lambda}$. These estimates follow directly from Corollary \ref{Corol-Strichartz} $(a)$ and Lemma \ref{maxfunc} $(a)$ in the Appendix \ref{appendix}.
\end{Rem}
~~\\
We now state the homogeneous estimates which follow from the Appendix \ref{appendix}.
\begin{Lemma}[Linear estimates I]\label{linear-bounds}
	Let $ u_0, u_1 \in L^2(\R^d)$,~ $e \in \mathcal{M}$,~ $\lambda > 0$ dyadic with $  \supp(\hat{u}_0),~ \supp(\hat{u}_1) \subset A_{\lambda}^d \cap A_e  $.
	Then the solution $u$ of \eqref{equation-linear} with $F=0$ satisfies
	\begin{align}\label{hom}
		\norm{u}{L_{e}^p L_{t, e^{\perp}}^q}  \leq C \lambda^{ \f{d}{2} - \f{1}{p}- \f{(d+1)}{q}} \left( \norm{u_0}{L^2} + \norm{u_1}{\dot{H}^{-2}} \right),
	\end{align}
	where $ (p,q)$ is an admissible pair. Further if $d \geq 3 $ and $ \hat{u}_0,~ \hat{u}_1$ have Fourier support in $A_{\lambda}^d$, then the solution $u$ of \eqref{equation-linear} with $F=0$ satisfies
	\begin{align}\label{hom2}
		\sup_{e \in \mathcal{M}}\norm{u}{L^2_e L^{\infty}_{t, e^{\perp}}} &\leq C \lambda^{\f{d-1}{2} } (\norm{u_0}{L^2_{x}} + \norm{u_1}{\dot{H}^{-2}_x}).\\\label{another-one}
		\norm{u}{L_{t}^p L_{x}^q}  &\leq C \lambda^{ \f{d}{2} - \f{2}{p}- \f{d}{q}} \left( \norm{u_0}{L^2} + \norm{u_1}{\dot{H}^{-2}} \right).
	\end{align}
\end{Lemma}
\begin{proof}
	By \eqref{Duhamel}, we note
	\[ u(t) = \f12e^{- i t \Delta} ( u_0 - i (- \Delta)^{-1} u_1) + \f12e^{ i t \Delta} ( u_0 + i (- \Delta)^{-1} u_1),\]
	hence \eqref{hom} follows from Corollary \ref{Corol-Strichartz} and \eqref{hom2} follows from Lemma \ref{maxfunc}. Estimate \eqref{another-one} is the classical Strichartz estimate for the Schr\"odinger group, for which we refer to Corollary \ref{Strichh}.
\end{proof}

\begin{Lemma}[Linear estimates II]\label{linear-bounds-inhomogeneous} For $ e \in \mathcal{M}$ and $\lambda > 0$ a dyadic number let $ F \in L^1_t L^2_x$ be localized  in $ A_{\lambda} \cap B_e $. Then the solution $u$ of \eqref{equation-linear} with $u_0 = u_1 = 0$ satisfies
	\begin{align}
		\norm{u}{L^p_{e}L^q_{t, e^{\perp}}}  &\lesssim \lambda^{ (d+1) ( \f{1}{2} - \f{1}{q}) - \f{1}{p} - \f52}\norm{F}{ L^1_t L^2_x}, \label{inhom1}\\
		\sup_{ \tilde{e} \in \mathcal{M}} \big( \norm{u}{L^2_{\tilde{e}}L^{\infty}_{t, \tilde{e}^{\perp}}} \big) &\lesssim \lambda^{\f{d}{2} -\f52}\norm{F}{ L^1_t L^2_x},\label{inhom3}
	\end{align}
	where $ (p,q)$ is an admissible pair. If $ \hat{F} $ has support in $ A_{\lambda}$, then the solution $u$ of \eqref{equation-linear} with $u_0 = u_1 = 0$ satisfies
	\begin{align}
		\norm{u}{L^p_{t}L^q_{x}}  \lesssim \lambda^{ d ( \f{1}{2} - \f{1}{q}) - \f{2}{p} - 2}\norm{F}{L^1_t L^2_x}\label{inhom2},
	\end{align}
	where $ (p,q)$ is an admissible pair.
\end{Lemma}
\begin{proof} 
	The estimate \eqref{inhom2} is the classical Strichartz estimate, which is stated in Corollary \eqref{Strichh}. For the remaining bounds \eqref{inhom1}, \eqref{inhom3}, we decompose the solution 
	\begin{align*}
		u(t) = \int_0^t \f{\sin(-(t-s)\Delta)}{ (-\Delta)} F(s)~ds =& \f{1}{2i}\int_0^t e^{-i(t-s) \Delta} (- \Delta)^{-1}F(s)~ds\\
		&~~+  \f{1}{2i}\int_0^t e^{i(t-s) \Delta} (- \Delta)^{-1}F(s)~ds.
	\end{align*}
	Especially, we have the pointwise bound
	$$  \left | \int_0^t e^{\pm i (t-s) \Delta} ( - \Delta)^{-1}F(s)~ds \right | \leq \int_0^{\infty} |e^{\pm i (t-s) \Delta}  ( - \Delta)^{-1}F(s)|~ds, $$
	and observe \eqref{inhom1} and  \eqref{inhom3} by Corollary \ref{Corol-Strichartz} $(a)$, Lemma \ref{maxfunc} $(a)$. If  $X$ denotes either one of the spaces on the LHS of \eqref{inhom2} and \eqref{inhom3}, we estimate
	\begin{align*}
		\norm{\int_{- \infty}^{t} e^{\pm i(t-s) \Delta} (- \Delta)^{-1}F(s) ~ds}{X} &\leq \int_{- \infty}^{\infty} \big \|e^{-
			\pm i(t-s) \Delta} (- \Delta)^{-1}F(s)\big \|_{X} ~ds\\
		&\lesssim \int_{- \infty}^{\infty}  \|(- \Delta)^{-1}F(s) \|_{L^2_{x}} ~ds.
	\end{align*} Here we note that in order to use the Corollary and the Lemma, we verify that   $ e^{\mp i s \Delta} ( - \Delta)^{-1}F(s) $ has Fourier support (in $ \xi$) in $(A^d_{\lambda} \cup A^d_{\lambda/2}) \cap A_e$ for all $ s \in \R $. This follows since $F$ is localized on $B_e \cap A_{\lambda} $ and hence also implies for normalized frequencies $ \lambda \sim 1 $
	$$ 	\norm{( - \Delta)^{-1}F}{L^1_t L^2_x} \lesssim \norm{F}{L^1_t L^2_x}. $$
\end{proof}
The next lemma follows from the homogeneous estimates in Lemma \ref{linear-bounds}, resp. Corollary \ref{Corol-Strichartz} and Lemma \ref{maxfunc}.
\begin{Lemma}[Trace estimate] \label{trans-lemma} Let $ F \in X^{\f12, 1}_{\lambda} $for a dyadic number $ \lambda$.
	\begin{itemize}
		\item[(a)] There holds
		\begin{align}\label{trans1}
			\sup_{e \in \mathcal{M}} \big( \norm{F}{L^2_e L^{\infty}_{t e^{\perp}} } \big)  \lesssim \lambda^{ \f{{d-1}}{2}} \norm{F}{X^{\f12, 1}_{\lambda}}\\
			\norm{F}{L^p_t L^{q}_{x }}  \lesssim \lambda^{\f{d}{2}- \f{d}{q}- \f{2}{p}} \norm{F}{X^{\f12, 1}_{\lambda}},\label{trans2}
		\end{align}
		for any admissible pair $(p, q)$.
		\item[(b)] We additionally assume $ \hat{F}(\tau, \cdot) $ has support in $ A_e $ for some $ e \in \mathcal{M}$ and all $ \tau \in \R$. Then there holds
		\begin{align}
			\norm{F}{L^p_e L^{q}_{t, e^{\perp}} }  \lesssim  \lambda^{\f{d}{2}- \f{d+1}{q} - \f{1}{p}} \norm{F}{X^{\f12, 1}_{\lambda}},\label{trans3}
		\end{align}
		where $ (p,q)$ is an admissible pair,~ $ p \geq 2$.	
	\end{itemize}
\end{Lemma}
\begin{Rem}In the following, we often use the dual estimates of \eqref{trans2} - \eqref{trans3}, i.e.
	\begin{align*}
		& \norm{F}{X^{-\f12, \infty}_{\lambda}} \lesssim \lambda^{\f{d}{2}- \f{d}{q}- \f{2}{p}}\norm{F}{L^{p'}_t L^{q'}_{x }},\\
		&\norm{F}{X^{-\f12, \infty}_{\lambda}} \lesssim \lambda^{\f{d}{2}- \f{d+1}{q} - \f{1}{p}} \norm{F}{L^{p'}_e L^{q'}_{t, e^{\perp} }},
	\end{align*}
\end{Rem}
\begin{proof}[Proof of Lemma \ref{trans-lemma}]
	For $ F \in X^{\f12, 1}_{\lambda}$, we have the representation
	$$F = \sum_{ \mu \leq 4 \lambda^2} Q_{\mu}F + h,$$
	where $ Lh = 0$ as mentioned in the previous section. We want to use \eqref{decomposition} and \eqref{norm-decomposition}. However, here we split over $ sign(\tau)$ and  write
	\begin{align*}
		\sum_{\mu \in 2^{\Z}} Q_{\mu} F &= \sum_{\mu \in 2^{\Z}} \int \int e^{ix \cdot \xi + it \tau} \varphi(w(\tau, \xi)\slash \mu) \hat{F}(\tau, \xi)~d\tau~d\xi\\
		&= \sum_{\mu \in 2^{\Z}} \int \int \chi(s + \xi^2 > 0) e^{ix \cdot \xi + it( s + \xi^2)} \varphi(w(s + \xi^2, \xi)\slash \mu) \hat{F}(s + \xi^2, \xi)~ds~d\xi\\
		&~~ + \sum_{\mu \in 2^{\Z}} \int \int \chi(-s + \xi^2 > 0) e^{ix \cdot \xi + it( s - \xi^2)} \varphi(w( s- \xi^2 , \xi)\slash \mu) \hat{F}(s - \xi^2, \xi)~ds~d\xi\\
		&= \sum_{\mu \in 2^{\Z}} \int   e^{ it( s - \Delta)} h^+_{\mu}(s)~ds + \sum_{\mu \in 2^{\Z}} \int   e^{ it( s + \Delta)} h^-_{\mu}(s)~ds,
	\end{align*}
	where
	\begin{align*}
		h^{\pm}_{\mu}(s) = \chi\big\{ \mu / 2^{\f{3}{2}}\leq |s| \leq 2 \mu \big\}  \int e^{i x \cdot \xi} \chi(\pm s + \xi^2 > 0) \varphi(w( s \pm \xi^2, \xi)\slash \mu) \hat{F}(s \pm \xi^2, \xi)~d \xi,
	\end{align*}
	and we used
	\begin{align*}
		\mu/2 \leq w(\tau , \xi) = ||\tau| - \xi^2| \frac{|\tau| + \xi^2}{(\tau^2 + |\xi|^4)^{\f12}} \leq \sqrt{2}||\tau| - \xi^2| \leq \sqrt{2}w(\tau , \xi) \leq 2 \sqrt{2} \mu.
	\end{align*}
	Now we assume there holds $ \norm{e^{ i\theta}e^{\pm it\Delta} f}{X} \lesssim \norm{f}{L^2_x}$ for some space $X$ and all $ \theta, t \in \R $ , then 
	\begin{align*}
		\sum_{\mu } \sum_{\pm} \norm{\int   e^{ it( s \mp \Delta)} h^{\pm}_{\mu}(s)~ds}{X} &\lesssim \sum_{\mu } \sum_{\pm} \mu^{\f12} \left(\int \norm{h^{\pm}_{\mu}(s)}{L^2_x}^2~ds \right)^{\f12}\\
		& \sim  \sum_{\mu } \sum_{\pm} \mu^{\f12} \norm{\chi(\pm \tau > 0) \varphi(w(\pm \tau, \xi)\slash \mu) \hat{F}(\tau, \xi)}{L^2_{\xi,\tau}}\\
		&\lesssim \sum_{\mu } \mu^{\f12} \norm{Q_{\mu}F}{L^2_{t, x}}.
	\end{align*}
	Hence \eqref{trans2} follows from the Strichartz estimate for Schr\"odinger groups and Lemma \ref{maxfunc}, since (for the limiting dyadic block $ \mu = 0$ with $Lh = 0$) we have (see Lemma \ref{Strichh}, resp. Lemma \ref{linear-bounds})
	\begin{align}
		\norm{h}{X} \lesssim \norm{h(0)}{L^2} + \norm{\partial_th(0)}{\dot{H}^{-2}},
	\end{align}
	where $ X = \lambda^{\f{d}{2} - \f{2}{p} - \f{d}{q}} L^p_tL^q_x $. For \eqref{trans1}, \eqref{trans3}, we use the decomposition
	$$ F = Q_{\leq \f{\lambda^2}{16}}F + (1 -   Q_{\leq \f{\lambda^2}{16}})F.$$
	Then we check that calculating $h^{\pm}_{\mu}(s)$ in the above argument for $ Q_{\leq \f{\lambda^2}{16}}F $, the function 
	$$  \xi \mapsto \widehat{h^{\pm}_{\mu}(s)}(\xi) $$ has support in $A^d_{\lambda/2} \cup A^d_{\lambda}$ for all $ s \in \R $. Hence, following the argument with 
	$$ X = \lambda^{\f{d}{2} - \f{1}{p} - \f{(d+1)}{q}}L^p_eL^q_{t, e^{\perp}},~~~ X = \bigcap_{e} \lambda^{\f{d-1}{2}} L^2_e L^{\infty}_{t, e^{\perp}},$$
	we obtain \eqref{trans1}, \eqref{trans3} by Corollary \ref{Corol-Strichartz} and Lemma \ref{maxfunc} for $ Q_{\leq \f{\lambda^2}{16}}F $ on the LHS. For \eqref{trans3}, we further note that by assumption $h^{\pm}_{\mu}(s) $ localizes in $A_e$ for all $s \in \R $. The remaining estimates  for $ (1 -   Q_{\leq \f{\lambda^2}{16}})F $ are equivalent to
	\begin{align*}
		\| (1 -   Q_{\leq \f{\lambda^2}{16}})F\|_{X} \lesssim \lambda \|(1 -   Q_{\leq \f{\lambda^2}{16}})F\|_{L^2_{t, x}},
	\end{align*}
	which follow from Sobelev embedding (thus the restriction to $ p \geq 2 $). As above, we obtain the estimates for the $ L h = 0$ part of the limiting dyadic block $ \mu = 0$ by Lemma \ref{linear-bounds}.
\end{proof}

	\subsection{Function spaces}\label{subsec:function-spaces}
	We now define the dyadic building blocks of the function spaces $ Z^{\f{d}{2}},~ W^{\f{d}{2}}$ and use the convention 
	$$ \norm{\cdot}{\lambda B_{\lambda}} =  \lambda^{-1} \norm{\cdot}{B_{\lambda}}.$$
	We set
	\begin{align} Z_{\lambda} =   X_{\lambda} ^{\f{1}{2}, 1} +  Y_{\lambda},
	\end{align}
	where $ Y_{\lambda} $ is the closure of 
	\begin{align*} &\{ f \in \mathcal{S}~|~ \supp(\hat{f})\subset A_{\lambda},~ \norm{f}{Y_{\lambda}} < \infty~\},\\[5pt]
		&\norm{f}{Y_{\lambda}} = \lambda^{- 2} \norm{Lf}{L_t^1 L^2_x} + \norm{f}{L^{\infty}_t L^2_x},
	\end{align*} 
	and the norm of $Z_{\lambda}$ is given by
	\[ \norm{u}{Z_{\lambda}} = \inf_{ u_1 + u_2 = u}\big( \norm{u_1}{X^{\f12, 1}_{\lambda}} +  \norm{u_2}{Y_{\lambda}}\big). \]
	For the nonlinearity, we construct $ W_{\lambda} = L (Z_{\lambda}) $, i.e.
	\begin{align}
		W_{\lambda} = \lambda^2 \big( X_{\lambda}^{ - \f{1}{2}, 1} +  (L^1_t L^2_x)_{\lambda} \big)  
	\end{align}
	where $ (L^1_t L^2_x)_{\lambda} $ is the closure of 
	\begin{align*}   
		&\{ F \in \mathcal{S}~|~ \supp(\hat{F})\subset A_{\lambda},~ \norm{F}{L^1_t L^2_x}  < \infty~\},
	\end{align*}
	and 
	\[ \norm{F}{W_{\lambda}} = \lambda^{-2}\inf_{ F_1 + F_2 = F}\big( \norm{F_1}{X^{-\f12, 1}_{\lambda}} + \norm{F_2}{L^1_t L^2_x}\big). \]
	Then, we define
	\begin{align}
		&\norm{u}{Z} = \sum_{\lambda \in 2^{\Z}} \lambda^{ \f{d}{2} }  \norm{P_{\lambda}(D) u}{Z_{\lambda}},\\[2pt]
		&\norm{F}{W} = \sum_{\lambda \in 2^{\Z}} \lambda^{ \f{d}{2} } \norm{P_{\lambda}(D) F}{W_{\lambda}},
	\end{align}
	and 
	\begin{align}
		&\norm{u}{Z^s}^2 = \sum_{ \lambda \in 2^{\Z}} \lambda^{2s} \norm{P_{\lambda}(D) u}{Z_{\lambda}}^2~~\text{for}~ s > \f{d}{2}.\\[2pt]
		&\norm{F}{W^s}^2 = \sum_{ \lambda \in 2^{\Z}} \lambda^{2s} \norm{P_{\lambda}(D) F}{W_{\lambda}}^2~~\text{for}~ s > \f{d}{2}.
	\end{align}
	\subsubsection{Embeddings, linear estimates and continuous operator}\label{subsec:embeddings}
	In this section, we provide some useful embeddings and multiplier theorems concerning $ Z_{\lambda}$ and $ W_{\lambda}$. We also show that the solution of \eqref{equation-linear} satisfies $ u \in Z^s$ if $ Lu \in W^s $ with the correct initial regularity $ \dot{B}^{2,1}_{\f{d}{2}} \times \dot{B}^{2,1}_{\f{d}{2}-2}$ or $\dot{H}^s \times \dot{H}^{s-2}$, respectively for $ s > \f{d}{2}$.\\[4pt]
	At the end of this section, we  show that $ Z_{\lambda}$ bounds 
	$$\lambda^{-\f12}L^{\infty}_eL^2_{t, e^{\perp}},~~ \bigcap_e \lambda^{\f{d-1}{2}}L^2_e L^{\infty}_{t, e^{\perp}}, $$
	in a suitable sense. Therefore, we apply the following heurstic argument, used similarly in \cite{bejenaru1} for Schr\"odinger maps. For solving \eqref{equation-linear} by $ u =V(F)$ with $ u_0 = u_1 = 0$, we rely on the inhomogeneous Strichartz estimate Lemma \ref{linear-bounds-inhomogeneous}
	\begin{align*}
		\norm{V(F)}{\lambda^{-\f12}L^{\infty}_eL^2_{t, e^{\perp}}} \lesssim \lambda^{-2}\norm{F}{L^1_t L^2_x},~~~(\tau,  \xi) \in \supp( \widehat{P_0F})
	\end{align*}
	and otherwise on \emph{inverting the symbol} of $L$
	\begin{align*}
		V(F) = \mathcal{F}^{-1}\big( \frac{ \widehat{F}(\tau, \xi)}{\tau^2 - |\xi|^4}\big),~~ (\tau,  \xi) \in \supp( \widehat{(1- P_0)F}).
	\end{align*}
	
	We first consider the following Lemma, which clearifies how the the spaces $Z_{\lambda}, W_{\lambda} $ behave under modulation cut-off and is essentially from \cite{tataru1} (adapted to the paraboloid $ \tau^2 = |\xi|^4$).
	\begin{Lemma}\label{multiplierr}
		The following operator are continuous for $ 1 \leq p < \infty$ with norms that are uniformly bounded in $ \mu \leq 4\lambda^2$.
		\begin{align*}
			&(a)~~P_{\lambda, \leq \mu},~P_{\lambda}P_0  :~ L^p_t L^2_x \to L^p_t L^2_x,~~ \mu \leq 4 \lambda^2\\[4pt]
			&(b)~~(1- Q_{\leq \mu})P_{\lambda} :~ Y_{\lambda} \to \mu^{-1} L^1_t L^2_x,~~\mu \leq 4 \lambda^2.
		\end{align*}
		
	\end{Lemma}
	\begin{proof}
		We first recall the following version of Tataru's argument in \cite{tataru1}, stated in \cite[chapter 2.4, Lemma 2.8]{geba-grillakis}.\\[4pt]
		Let  $ C > 0 $ and $ M = \mathcal{F}^{-1}( m ( \tau, \xi)\mathcal{F}(\cdot)) $ be a Fourier multiplier such that the following holds.
		\begin{itemize}
			\setlength\itemsep{0.5em}
			\item[(i)] For any $ \xi$, there holds  $  \supp( \tau \mapsto m(\tau, \xi) ) \subset A_{\xi}$, where $ A_{\xi} $ has  measure $ \leq C$.
			\item[(ii)] For $ N \geq 2 $ there exists $ C_N > 0 $ such that 
			$$ \norm{m}{L^{\infty}_{\tau, \xi}} + C^N \norm{\partial_{\tau}^N m(\tau, \xi)}{L^{\infty}_{\tau, \xi}} \leq C_N. $$
		\end{itemize}
		Then the operator
		\begin{align}
			M : L^p_t L^2_x \to L^p_t L^2_x,~~ 1 \leq p \leq \infty,
		\end{align}
		is continuous and $ \norm{M}{} \lesssim C_N$.\\[4pt]
		By Plancherel (in $ \xi$) and Young's inequality (in $t$), it suffices to proof $ K \in L^1_t L^2_{\xi}$, where
		$$ K(t, \xi) = \int e^{i t \tau} m(\tau, \xi)~d\tau. $$
		However by $(i),~(ii)$ it follows $ \norm{K}{L^1_tL^2_{\xi}} \lesssim C_N C $ and by $(ii)$ and integration by parts 
		$$ |K(t, \xi)| = \left| \f{(-1)^N}{|t|^N i^N} \int e^{i t \tau } \partial_{\tau}^N m(\tau, \xi ) ~d\tau \right| \lesssim \f{C_N}{C^N|t|^N }, $$
		by which
		$$ | K(t, \xi) | \lesssim \frac{C_N C}{(1 + C |t|)^N}.$$
		This argument applies, similar to \cite{tataru1}, to $ P_{\lambda, \leq \mu}$ with multiplier
		$$m_{\mu, \lambda}(\tau, \xi) = \sum_{\tilde{\mu} \leq \mu}\varphi(( \tau^2 + \xi^4)^{\f14}/\lambda) \varphi(w(\tau, \xi)/\tilde{\mu}), $$
		since there holds for $ N \in \N$  and $ \xi \in \R^d $ fixed
		\begin{align}\label{verify}
			&|\partial_{\tau}^N m_{\mu, \lambda}(\tau, \xi)| \lesssim_N \mu^{-N},~~ \supp( m_{\mu, \lambda})  \subset \{ (\tau, \xi)~|~ ||\tau| - \xi^2| \leq 4 \mu \}.
		\end{align}
		For the second operator 
		$$P_{\lambda}P_0 u = \mathcal{F}^{-1}( \varphi(( \tau^2 + \xi^4)^{\f14}/\lambda)) \chi(\tau, \xi)\hat{u}(\tau, \xi))$$
		in $(a)$, we note that $ \chi $ is invariant under scaling and hence the claim reduces to continuity of $ P_1 P_0 : L^p_t L^2_x \to L^p_t L^2_x $. This follows directly from the above argument.\\[3pt]
		Now for part $(b)$, we write
		\begin{align*}
			\mathcal{F} (1 - Q_{\leq \mu})P_{\lambda} u)(\tau, \xi) =&~ \big( 1 - \sum_{\tilde{\mu} \leq \mu}\varphi(w(\tau, \xi)/\tilde{\mu})\big) \varphi(( \tau^2 + \xi^4)^{\f14}/\lambda) \hat{u}(\tau, \xi)\\
			=&~  \mu^{-1} \lambda^{-2}\big( 1 - \sum_{\tilde{\mu} \leq \mu}\varphi(w(\tau, \xi)/\tilde{\mu})\big) \frac{\varphi(( \tau^2 + \xi^4)^{\f14}/\lambda) \mu \lambda^2}{w(\tau, \xi)( \tau^2 + \xi^4)^{\f12}} \widehat{Lu}(\tau, \xi)\\
			=:&~ \mu^{-1} \lambda^{-2} \tilde{m}_{\mu, \lambda}(\tau, \xi) \widehat{Lu}(\tau, \xi).
		\end{align*}
		It hence suffices to prove continuity of the operator $ \mathcal{F}^{-1}(\tilde{m}_{\mu, \lambda} \mathcal{F} (\cdot)) $ on $ L^1_t L^2_x $. This follows similarly as in the proof for the cone in \cite{tataru1}. We sketch the argument following the proof in  \cite[chapter 2.4]{geba-grillakis}.
		There holds
		\begin{align}\label{yo}
			||\tau| - \xi^2 | \tilde{m}_{\mu, \lambda}(\tau, \xi) + ||\tau| - \xi^2|^3 \partial_{\tau}^2\tilde{m}_{\mu, \lambda}(\tau, \xi) \lesssim \mu.
		\end{align} 
		Hence, considering the support 
		$$ \{ (\tau, \xi)~|~ ||\tau| - \xi^2| \geq \mu/\sqrt{2}, ~~~ |\tau| + \xi^2 \leq 4 \sqrt{2}\lambda^2~\},$$
		we infer
		\begin{align*}
			&\left | \int e^{it \tau} \tilde{m}_{\mu, \lambda}(\tau, \xi)~d \tau \right| \lesssim \mu \log(\lambda^2/\mu),~~~~~\left | t^2 \int e^{it \tau} \tilde{m}_{\mu, \lambda}(\tau, \xi)~d \tau \right| \lesssim \mu^{-1}.
		\end{align*}
		Integration gives boundedness of the following terms (uniform in $\mu, \lambda$)
		$$\norm{K}{L^1_t L^2_{\xi}} \lesssim \int_{|t| \leq \f{1}{4 \sqrt{2}  \lambda^2}} \norm{K(t, \cdot)}{L^{\infty}}~dt + \int_{|t| \geq  \frac{\sqrt{2}}{\mu}} \norm{K(t, \cdot)}{L^{\infty}}~dt + \int_{\f{1}{4 \sqrt{2}  \lambda^2} \leq |t| \leq \frac{\sqrt{2}}{\mu}}  \norm{K(t, \cdot)}{L^{\infty}}~dt.$$
		For the last term, we estimate
		\begin{align*}
			&\norm{K(t, \cdot)}{L^{\infty}}  \lesssim \int_{\frac{\mu}{\sqrt{2}} \leq ||\tau| - \xi^2| \leq \f{1}{|t|}} \f{\mu}{||\tau| - \xi^2|}~d\tau + \f{1}{t^2}\int_{ ||\tau| - \xi^2| \geq  \f{1}{|t|}} \f{\mu}{||\tau| - \xi^2|^3}~d\tau \lesssim \mu( 1 - \log(|t|\mu)),
		\end{align*}
		and hence
		$$
		\int_{\f{1}{4 \sqrt{2}  \lambda^2} \leq |t| \leq \frac{\sqrt{2}}{\mu}}  \norm{K(t, \cdot)}{L^{\infty}}~dt \lesssim 1.$$
	\end{proof}
	
\begin{Lemma}\label{emb-Lemma} 
	We have 
	\begin{align}\label{onee}
		&W_{\lambda} \subset \lambda^{3} L^{2}_{t,x} ,~~~Z_{\lambda} \subset  \lambda^{\f{d}{2}}L^{\infty}_{t,x}\\[3pt]
		\label{clear}
		&X^{\f12, 1}_{\lambda} \subset Z_{\lambda} \subset X^{\f12, \infty}_{\lambda}.
	\end{align}
\end{Lemma}
\begin{proof}
	For \eqref{clear}, we note that $ X^{\f12, 1}_{\lambda} \subset Z_{\lambda} $ follows by definiton and $ Z_{\lambda} \subset X^{\f12, \infty}_{\lambda}$ is proven as follows.\\[3pt]
	The norm of   $X^{\f12, \infty}_{\lambda}$  is estimated against the norm of the $ X^{\f12, 1}_{\lambda}$ part and further, for the $ L^1_t L^2_x$ part in $ Y_{\lambda}$, we deduce from Lemma \ref{trans-lemma} 
	\begin{align*}
		\norm{u_{\lambda}}{X^{\f12, \infty}_{\lambda}} &\lesssim \lambda^{-2} \norm{Lu_{\lambda}}{X^{-\f12, \infty}_{\lambda}} +  \norm{u(0)}{L^2_x} + \norm{\partial_tu(0)}{H^{-2}_x}\\
		&\lesssim \lambda^{- 2}\norm{Lu_{\lambda}}{L^1_t L^2_{x}} + \norm{u_{\lambda}}{L^{\infty}_t L^2_x},
	\end{align*}
	which reads as
	\begin{align*}
		\norm{u_{\lambda}}{X^{\f12, \infty}_{\lambda}} \lesssim \norm{u_{\lambda}}{Y_{\lambda}},~~ u_{\lambda} \in Y_{\lambda}.
	\end{align*}
	Concerning \eqref{onee} in the Lemma, we note
	\begin{align*}
		\norm{u_{\lambda} }{L^2_{t,x}} \lesssim \lambda \norm{u_{\lambda}}{L^1_t L^2_{x}} \sim \lambda^{ 3} \norm{u_{\lambda}}{\lambda^2 L^1_t L^2_x },~~ u_{\lambda} \in L^1_t L^2_x,
	\end{align*}
	where we used that $ \hat{u}_{\lambda}(\cdot, \xi)$ is localized (in $\tau$) on an interval  on length $\sim \lambda^2$. Hence, since also,
	\begin{align*}
		&\norm{u_{\lambda} }{L^{2}_{t,x}} \lesssim \lambda \sum_{\mu \lesssim \lambda^2} \mu^{- \f12} \norm{Q_{\mu}(u_{\lambda})}{L^2_{t,x}},~~~ u_{\lambda} \in X^{-\f12, 1}_{\lambda},
	\end{align*}
	we obtain the first claim. For the $ L^{\infty}_{t,x}$ embedding, we estimate similarly by Lemma \ref{trans-lemma}
	\begin{align*}
		&\norm{u_{\lambda}}{L^{\infty}_{t,x}} \lesssim \lambda^{\f{d}{2}} \norm{u_{\lambda}}{X^{\f12,1}_{\lambda}}.
	\end{align*}
	For the $ Y_{\lambda}$ part, we obtain by a direct application of the classical Strichartz estimate 
	\begin{align*}
		\norm{u_\lambda}{L^{\infty}_{t,x}} \lesssim \norm{u(0)}{\dot{H}^{\f{d}{2}}} + \norm{\partial_{t}u(0)}{\dot{H}^{\f{d}{2}-2}} + \lambda^{\f{d}{2} - 2}\norm{Lu}{L^1_t L^2_x}
		\lesssim \lambda^{\f{d}{2}} \norm{u_{\lambda}}{Y_{\lambda}}.
	\end{align*}
	from Lemma \ref{linear-bounds} and Lemma \ref{linear-bounds-inhomogeneous} for $ p = q = \infty$. 
\end{proof}

\begin{Propo}\label{Linea-propo}There holds
	\begin{align}\label{first-two-claimes1}
		&Z^{\f{d}{2}} \subset C(\R, \dot{B}^{2,1}_{\f{d}{2}}) \cap \dot{C}^1(\R, \dot{B}^{2,1}_{\f{d}{2}-2})\\\label{first-two-claimes2}
		&Z^{s} \subset C(\R, \dot{H}^s) \cap \dot{C}^1(\R, \dot{H}^{s-2})
	\end{align}
	Further we have
	\begin{align}\label{This-claim}
		&\norm{u}{Z^{\f{d}{2}}} \lesssim \norm{(u(0), \partial_tu(0))}{\dot{B}^{2,1}_{\f{d}{2}} \times \dot{B}^{2,1}_{\f{d}{2}-2}} + \norm{Lu}{W^{\f{d}{2}}},\\ \label{that-claim}
		&\norm{u}{Z^s} \lesssim \norm{(u(0), \partial_tu(0))}{\dot{H}^{s} \times \dot{H}^{s-2}} + \norm{Lu}{W^{s}},~~~ s > \frac{d}{2},\\ \label{last-claim}
		&\norm{Lu}{W^{\f{d}{2}}} \lesssim \norm{u}{Z^{\f{d}{2}}},~~\norm{Lu}{W^s} \lesssim  \norm{u}{Z^{s}},~~~ s > \frac{d}{2}.
	\end{align}
\end{Propo}
\begin{proof}
	The claim \eqref{last-claim} follows from the definition of $ Z_{\lambda}, W_{\lambda}$ since $ \lambda^2 L^1_t L^2_x = L Y_{\lambda}$  and for the $X^{\f12, p}_{\lambda}$ part, we use 
	$$ \norm{Lu}{X^{-\f12,1 }_{\lambda}} \lesssim \lambda^2 \norm{u}{X^{\f12,1 }_{\lambda}}.$$ For \eqref{first-two-claimes1} and \eqref{first-two-claimes2}, if suffices to show
	\begin{align*}
		\norm{P_{\lambda}(D)u}{L^{\infty}_t\dot{B}^{2,1}_{\f{d}{2}}} + \norm{P_{\lambda}(D)\partial_tu(t)}{L^{\infty}_t\dot{B}^{2,1}_{\f{d}{2}-2}} \lesssim \lambda^{\f{d}{2}}\norm{P_{\lambda}(D)u}{Z_{\lambda}},~
	\end{align*}
	where by Bernstein
	\begin{align}\norm{P_{\lambda}(D)\partial_tu(t)}{L^{\infty}_t\dot{B}^{2,1}_{\f{d}{2}-2}} \lesssim 	\norm{P_{\lambda}(D)u}{L^{\infty}_t\dot{B}^{2,1}_{\f{d}{2}}}.
	\end{align}
	Then, since 
	\begin{align}
		\norm{P_{\lambda}(D)u}{L^{\infty}_t\dot{B}^{2,1}_{\f{d}{2}}} \leq \sum_{\tilde{\lambda} \leq \lambda} \big(\tilde{\lambda}\slash\lambda\big)^{\f{d}{2}} \lambda^{\f{d}{2}} \norm{P_{\lambda}(D) P_{\tilde{\lambda}}(\nabla)u}{L^{\infty}_tL^2_{x}},
	\end{align}
	the embedding and the continuity in time follow from $ Z_{\lambda} \subset S_{\lambda} \subset C_tL^2_x$ and we proceed similarly for the emdedding of $ Z^s$ using square sums. Now for \eqref{This-claim} and \eqref{that-claim}, we use Duhamel's formula 
	$$ u = S(u(0), \partial_tu(0)) + V(Lu),$$
	where $ S(u_0, u_1) $ solves \eqref{equation-linear} for $F=0$ and $ VF$ solves $ \eqref{equation-linear} $ for $ u_0 = u_1 = 0$. The homogeneous solution is estimated by the Strichartz bound in Lemma \ref{linear-bounds} in the energy case $ p = \infty, q = 2 $. This is also directly verified by 
	$$ \mathcal{F}_x(P_{\lambda}(D)S(u(0), \partial_tu(0)))(t, \xi) = \varphi(2^{\f14} |\xi| / \lambda)(\cos( |\xi|^2 t)\widehat{u(0)}(\xi) + |\xi|^{-2} \sin( |\xi|^2 t) \widehat{\partial_tu(0)}(\xi)),  $$
	and hence
	$$ \norm{P_{\lambda} S(u(0), u_t(0))}{Z_{\lambda}} \lesssim \norm{P_{\lambda} S(u(0), u_t(0))}{L^{\infty}_t L^2_x} \lesssim \norm{u_{\lambda}(0)}{L^2} + \norm{\partial_t u_{\lambda}(0)}{H^{-2}}.$$
	For the inhomogeneous solution $ V(Lu)$ we estimate the $ X^{\f12, p}_{\lambda}$ part by
	\begin{align}
		\norm{V(Lu_{\lambda})}{X^{\f12, 1}_{\lambda} } \lesssim \lambda^{-2} \norm{Lu_{\lambda}}{X^{-\f12, 1}_{\lambda}}, 
	\end{align}
	and for $Y_{\lambda}$, we use Lemma \ref{linear-bounds-inhomogeneous} in order to conclude
	\begin{align*}
		\norm{V(Lu_{\lambda})}{Y_{\lambda}} = \lambda^{-2}\norm{Lu_{\lambda}}{L^1_t L^2_x} + \norm{u_{\lambda}}{L^{\infty}_t L^2_x}\lesssim   \lambda^{-2}\norm{Lu_{\lambda}}{L^1_t L^2_x}.
	\end{align*}
\end{proof}
We further estimate the lateral Strichartz norm  and establish the maximal function estimate.
\begin{Propo}\label{emb-propo-2}
	For any dyadic number $ \lambda \in 2^{\Z} $ we have
	\begin{align}
		Z_{\lambda} &\subset S_{\lambda} \cap \sum_{e \in \mathcal{M}} S_{\lambda}^e, \label{embed1}\\
		Z_{\lambda} &\subset \bigcap_{e \in \mathcal{M}} \lambda^{\f{n-1}{2}} L^2_e L^{\infty}_{t, e^{\perp}}.\label{embed2}
	\end{align}
\end{Propo}
where $ S_{\lambda}^e  $ is the closure of
\begin{align*} \bigg\{ f \in \mathcal{S}~|~ \supp(\hat{f})\subset A_{\lambda},~ \norm{f}{S_{\lambda}^e} = \sup_{(p,q)} \big( \lambda^{\f{1}{p} + \f{(d+1)}{q} - \f{d}{2} }\norm{f}{L^p_{e}L^q_{t, e^{\perp}}} \big) < \infty \bigg\}
\end{align*}
with $(p,q)$ ranging over all admissible pairs with $ p \geq 2$.
\begin{proof}
	For  \eqref{embed1}, we first consider the embedding $ Z_{\lambda} \subset S_{\lambda} $. Thus, the $X^{\f12, 1}_{\lambda}$ part satisfies for any admissible pair $(p,q)$
	$$ \lambda^{\f{2}{p} + \f{d}{q} - \f{d}{2}} \norm{u_{\lambda}}{L^p_t L^q_x} \lesssim \norm{u_{\lambda}}{X^{\f12, 1}_{\lambda}},$$
	by Lemma \ref{trans-lemma}. Likewise, we obtain the same bound against the $Y_{\lambda} $ part by Lemma \ref{linear-bounds} and Lemma \ref{linear-bounds-inhomogeneous}.
	For the $S_{\lambda}^e$ embedding, we decompose as follows
	\begin{align}\label{decompositionn}
		u_{\lambda} = \sum_{ e \in \mathcal{M}} u^e_{\lambda},~~ u_{\lambda}^e = P_e(\nabla) u_{\lambda},
	\end{align}
	which suffices to obtain \eqref{embed1} for the $X^{\f12, 1}_{\lambda}$ part directly from  Lemma \ref{trans-lemma}. Now, considering the $Y_{\lambda}$ part of $Z_{\lambda}$, we further write
	\begin{align*}
		u_{\lambda }^e = P_0 u_{\lambda}^e +  (1 - P_0) u_{\lambda}^e.
	\end{align*}
	Then, $P_0 u_{\lambda}^e $ is localized in $ B_e $ and (by definition of $ P_0,~ 1 - P_0$)
	\begin{align*}
		P_0 u_{\lambda}^e &= S(u^e_{\lambda}(0), \partial_tu^e_{\lambda}(0)) +  V( P_0L(u_{\lambda}^e)),~~~(1 - P_0) u_{\lambda}^e = V( (1 - P_0)L(u_{\lambda}^e)).
	\end{align*}
	Hence  by Lemma \ref{linear-bounds} and Lemma \ref{linear-bounds-inhomogeneous} we have
	\begin{align}
		\lambda^{\f{1}{p} + \f{(d+1)}{q} - \f{d}{2} }\norm{P_0u_{\lambda}^e}{L^p_{e}L^q_{t, e^{\perp}}} &\lesssim  \lambda^{- 2} \norm{P_0 Lu_{\lambda}^e}{L^1_t L^2_{x}} + \norm{u_{\lambda}^e(0)}{L^2} + \norm{\partial_tu_{\lambda}^e(0)}{H^{-2}}\\ \nonumber
		&\lesssim \norm{u_{\lambda}}{Y_{\lambda}^e},
	\end{align}
	by Lemma \ref{multiplierr} and continuity of $ P_e(\nabla) $ on $L^1_t L^2_x$. Similarly, by Lemma \ref{trans-lemma}, we infer
	\begin{align*}
		\lambda^{\f{1}{p} + \f{(d+1)}{q} - \f{d}{2} }\norm{(1 - P_0)u_{\lambda}^e}{L^p_{e}L^q_{t, e^{\perp}}} \lesssim  \norm{V(1 - P_0) (Lu_{\lambda}^e)}{X^{\f12, 1}_{\lambda} }
		& \lesssim  \lambda^{-2}\norm{(1 - P_0) Lu_{\lambda}^e}{X^{-\f12, 1}_{\lambda} }\\
		& \lesssim \lambda^{-2} \norm{(1 - P_0) Lu_{\lambda}^e}{X^{-\f12, \infty}_{\lambda} }\\
		& \lesssim \norm{u_{\lambda}}{Y_{\lambda}^e},
	\end{align*}
	where we used  $ (1 - P_0)X^{\f12, 1}_{\lambda} \sim (1 - P_0)X^{\f12, \infty}_{\lambda}$ uniform in the frequency $ \lambda \in 2^{\Z}$ and the dual trace inequelity from Lemma \ref{trans-lemma} in the last step. Hence we sum over $ e \in \mathcal{M}$ and take  the infimum over $ u_{\lambda} = \sum_e u_{\lambda}^e$ with $ u_{\lambda}^e \in S_{\lambda}^e$.
	The $ L^2_e L^{\infty}_{t, e^{\perp}} $ embedding \eqref{embed2} follows similarly using Lemma \ref{trans-lemma}, the decomposition \eqref{decompositionn} and Lemma \ref{linear-bounds}, \ref{linear-bounds-inhomogeneous}. Especially
	\begin{align}
		\sup_{\tilde{e}}\big(\lambda^{\f{1-d}{2}}\norm{P_0u_{\lambda}^e}{L^2_{\tilde{e}}L^{\infty}_{t, \tilde{e}^{\perp}}}\big) &\lesssim \norm{u_{\lambda}}{Y_{\lambda}},\\
		\sup_{\tilde{e}}\big(\lambda^{\f{1-d}{2}}\norm{(1 - P_0)u_{\lambda}^e}{L^2_{\tilde{e}}L^{\infty}_{t, \tilde{e}^{\perp}}}\big) &\lesssim \lambda^{-2}\norm{Lu_{\lambda}}{X^{ - \f12, \infty}_{\lambda}} \lesssim \norm{u_{\lambda}}{Y_{\lambda}}
	\end{align}
	Again the estimate for the $X^{\f12, 1}_{\lambda}$ part follows directly by Lemma \ref{trans-lemma}.
\end{proof}
\section{Bilinear estimates}\label{sec:multi}
For the bilinear interaction, we  write 
\begin{align}
	u \cdot v =& \sum_{\lambda_1, \lambda_2, \lambda} (u_{\lambda_1} v_{\lambda_2})_{\lambda} \nonumber \\
	=& \sum_{\lambda_2 \gg \lambda_1 } \big[ (u_{\lambda_1} v_{\lambda_2})_{\lambda_2/2} + (u_{\lambda_1} v_{\lambda_2})_{\lambda_2} + (u_{\lambda_1} v_{\lambda_2})_{2\lambda_2} \big]\label{yes}\\
	& + \sum_{\lambda_1 \gg \lambda_2 } \big[ (u_{\lambda_1} v_{\lambda_2})_{lambda_2/2} + (u_{\lambda_1} v_{\lambda_2})_{\lambda_1} + (u_{\lambda_1} v_{\lambda_2})_{2\lambda_1} \big]\\
	& + \sum_{|\log_2(\lambda_1/\lambda_2)|\sim 1} \sum_{ \lambda \lesssim \max \{\lambda_1, \lambda_2\}} (u_{\lambda_1} v_{\lambda_2})_{\lambda}.\label{yyes}
\end{align}
Due to symmetry, we restrict \eqref{yes} - \eqref{yyes} to
\begin{align*}
	&\sum_{\lambda_2 \gg \lambda_1 } \big[ (u_{\lambda_1} v_{\lambda_2})_{\lambda_2/2} + (u_{\lambda_1} v_{\lambda_2})_{\lambda_2} + (u_{\lambda_1} v_{\lambda_2})_{2\lambda_2} \big]\\
	&+ \sum_{\lambda_1 \sim \lambda_2} \sum_{ \lambda \lesssim  \lambda_2} (u_{\lambda_1} v_{\lambda_2})_{\lambda},
\end{align*}
and thus further reduce to the  interactions
\[ \lambda_1 \ll \lambda_2:~~~(u_{\lambda_1} v_{\lambda_2})_{\lambda_2},~~~\text{and}~~~\lambda_1 \leq \lambda_2:~~~(u_{\lambda_2} v_{\lambda_2})_{\lambda_1}. \]
\begin{Lemma}\label{algebra}
	\begin{align}
		&(a)~~~\norm{u_{\lambda_1}v_{\lambda_2}}{Z_{\lambda_2}} \lesssim \lambda_1^{\f{d}{2}} \norm{u_{\lambda_1}}{Z_{\lambda_1}} \norm{v_{\lambda_2}}{Z_{\lambda_2}},~~ \lambda_1 \ll \lambda_2\\[3pt]
		&(b)~~~\norm{(u_{\lambda_2}v_{\lambda_2})_{\lambda_1}}{Z_{\lambda_1}} \lesssim \lambda_2^{\f{d}{2}} \norm{u_{\lambda_2}}{Z_{\lambda_2}} \norm{v_{\lambda_2}}{Z_{\lambda_2}},~~\lambda_1 \leq \lambda_2. 
	\end{align}
\end{Lemma}
\begin{proof}
	For part $(a)$, we decompose 
	\begin{align}\label{deco}
		(u_{\lambda_1}v_{\lambda_2})_{\lambda_2} = Q_{\leq 4 \lambda_1 \lambda_2 }(u_{\lambda_1}v_{\lambda_2})_{\lambda_2} +  (1 - Q_{\leq 4 \lambda_1 \lambda_2})(u_{\lambda_1}v_{\lambda_2})_{\lambda_2},
	\end{align} 
Here splitting the modulation by  $ \mu = \lambda_1 \lambda_2$ (instead of e.g. the natural choice $ \lambda_1^2$) is necessary in order to handle 
$$	L(u_{\lambda_1} ( 1 - Q_{ \leq  \lambda_1 \lambda_2})v_{\lambda_2})_{\lambda_2} \in L^1_t L^2_x $$
and specifically the  worst interaction $ \nabla_xu_{\lambda_1} \cdot \nabla_x^3 v_{\lambda_2}$ which is done below. The smoothing is then exploited via 
$$ L^{2}_e L^{\infty}_{t, e^{\perp}} \cdot  L^{\infty}_e L^2_{t, e^{\perp}} \subset L^2$$
for the term $  Q_{\leq 4 \lambda_1 \lambda_2 }(u_{\lambda_1}v_{\lambda_2})_{\lambda_2}$ as follows . First, we place $ Q_{\leq 4 \lambda_1 \lambda_2 }(u_{\lambda_1}v_{\lambda_2})_{\lambda_2} \in X^{\f12, 1}_{\lambda_2}$ by estimating
	\begin{align}\label{hilfs}
		\norm{(u_{\lambda_1}v_{\lambda_2})_{\lambda_2}}{L^2_{t,x}} \lesssim \lambda_1^{\f{d-1}{2}} \lambda_2^{- \f12}\norm{u_{\lambda_1}}{Z_{\lambda_1}} \norm{v_{\lambda_2}}{Z_{\lambda_2}}.
	\end{align}
	Then, from $ X_{\lambda_2}^{\f12, 1} \subset Z_{\lambda_2}$, \eqref{hilfs} gives
	\begin{align*}
		\norm{ Q_{\leq 4 \lambda_1 \lambda_2 }(u_{\lambda_1}v_{\lambda_2})_{\lambda_2}}{Z_{\lambda_2}}
		&\lesssim 
		\norm{ Q_{\leq 4 \lambda_1 \lambda_2 }(u_{\lambda_1}v_{\lambda_2})_{\lambda_2}}{X_{\lambda_2}^{\f12, 1}}\\ 
		&\lesssim 
		\big(\sum_{\mu \leq 4 \lambda_1 \lambda_2} \mu^{\f12} (\lambda_1 \lambda_2)^{-\f12} \big) \lambda_1^{\f{d}{2}} \norm{u_{\lambda_1}}{Z_{\lambda_1}} \norm{v_{\lambda_2}}{Z_{\lambda_2}}.
	\end{align*}
	For \eqref{hilfs}, we write $ u_{\lambda_1} v_{\lambda_2} =  \sum_{e \in \mathcal{M}} u_{\lambda_1}v_{\lambda_2}^e$ where $ v_{\lambda_2}^e \in S_{\lambda_2}^e$. Hence
	\begin{align}
		\norm{(u_{\lambda_1} v_{\lambda_2}^e)_{\lambda_2}}{L^2_{t,x}} &\leq \norm{u_{\lambda_1}}{L^{2}_e L^{\infty}_{t, e^{\perp}} } \norm{v_{\lambda_2}^e}{L^{\infty}_e L^2_{t, e^{\perp}}}\\ \nonumber
		& \leq \lambda_1^{\f{d-1}{2}} \norm{u_{\lambda_1}}{\lambda_1^{\f{d-1}{2}} \bigcap_{\tilde{e}} L^{2}_{\tilde{e}} L^{\infty}_{t, \tilde{e}^{\perp}} }   \lambda_2^{- \f12}\norm{v_{\lambda_2}^e}{\lambda_2^{- \f12} L^{\infty}_e L^2_{t, e^{\perp}}}.
	\end{align}
	Summing over $ e \in \mathcal{M}$, the claim follows from Proposition \ref{emb-propo-2}. Secondly, we note for $   \lambda_1^2 \ll \mu $
	\begin{align}\label{iden-mod}
		Q_{\mu}(u_{\lambda_1}v_{\lambda_2})_{\lambda_2} = Q_{\mu}\big(u_{\lambda_1} \sum_{ |j| \leq 2} Q_{2^j\mu}v_{\lambda_2}\big) 
	\end{align}
	Hence we write 
	\begin{align*}
		(1 - Q_{\leq 4 \lambda_1 \lambda_2})(u_{\lambda_1}v_{\lambda_2})_{\lambda_2} = (1 - Q_{\leq 4 \lambda_1 \lambda_2})(u_{\lambda_1}(1 - Q_{\leq  \lambda_1 \lambda_2})v_{\lambda_2})_{\lambda_2}
	\end{align*}
	In order to estimate the remaining part in \eqref{deco}, using Lemma \ref{multiplierr}, it thus suffices to prove
	\begin{align}\label{nrone}
		&\norm{(u_{\lambda_1}(1 - Q_{\leq  \lambda_1 \lambda_2})v_{\lambda_2})_{\lambda_2}}{X_{\lambda_2}^{\f12, 1}} \lesssim \norm{u_{\lambda_1}}{Z_{\lambda_1}} \norm{v_{\lambda_2}}{X_{\lambda_2}^{\f12,1}}\\\label{nrtwo}
		&\norm{(u_{\lambda_1}(1 - Q_{\leq  \lambda_1 \lambda_2})v_{\lambda_2})_{\lambda_2}}{Y_{\lambda_2}} \lesssim \norm{u_{\lambda_1}}{Z_{\lambda_1}} \norm{v_{\lambda_2}}{Y_{\lambda_2}}.
	\end{align}
	The estimate \eqref{nrone} and the $ L^{\infty}_t L^2_x$ summand of \eqref{nrtwo} follow from the embedding $ Z_{\lambda_1}  \subset \lambda_1^{\f{d}{2}}L^{\infty}_{t,x}$ by factoring off the $L^{\infty}_{t,x}$ norm of $u_{\lambda_1}$. For the second estimate \eqref{nrtwo}, we further calculate
	\begin{align*}
		L(u_{\lambda_1} ( 1 - Q_{ \leq  \lambda_1 \lambda_2})v_{\lambda_2})_{\lambda_2} =&~ u_{\lambda_1} L ( 1 - Q_{\leq  \lambda_1 \lambda_2})v_{\lambda_2}+ \partial_t u_{\lambda_1} \partial_t ( 1 - Q_{\leq  \lambda_1 \lambda_2})v_{\lambda_2}\\[3pt]
		&+ \partial_t^2 u_{\lambda_1}  ( 1 - Q_{\leq  \lambda_1 \lambda_2})v_{\lambda_2}
		+ \Delta^2( u_{\lambda_1} ( 1 - Q_{\leq  \lambda_1 \lambda_2})v_{\lambda_2})\\[3pt]
		& -  u_{\lambda_1} \Delta^2( 1 - Q_{\leq  \lambda_1 \lambda_2})v_{\lambda_2}),
	\end{align*}
	hence we estimate
	\begin{align*}
		\|L(u_{\lambda_1} &( 1 - Q_{ \leq  \lambda_1 \lambda_2})v_{\lambda_2})_{\lambda_2}\|_{L^1_t L^2_x}\\
		&\lesssim \norm{u_{\lambda_1} L ( 1 - Q_{\leq  \lambda_1 \lambda_2})v_{\lambda_2}}{L^1_t L^2_x} + \norm{\partial_t u_{\lambda_1} \partial_t ( 1 - Q_{\leq  \lambda_1 \lambda_2})v_{\lambda_2}}{L^1_t L^2_x}\\
		&~~~~+ \norm{\partial_t^2 u_{\lambda_1}  ( 1 - Q_{\leq  \lambda_1 \lambda_2})v_{\lambda_2}}{L^1_t L^2_x}\\
		&~~~~ + \norm{\Delta^2( u_{\lambda_1} ( 1 - Q_{\leq  \lambda_1 \lambda_2})v_{\lambda_2}) -  u_{\lambda_1} \Delta^2( 1 - Q_{\leq  \lambda_1 \lambda_2})v_{\lambda_2})}{L^1_t L^2_x}.
	\end{align*}
	Calculating the expression in the latter norm and factoring off the derivatives of $u_{\lambda_1}$ in $ L^{\infty}$, we infer (using Bernstein's inequality) 
	\begin{align*}
		\|L(u_{\lambda_1} &( 1 - Q_{ \leq  \lambda_1 \lambda_2})v_{\lambda_2})_{\lambda_2}\|_{L^1_t L^2_x}\\
		&\lesssim \norm{u_{\lambda_1} L ( 1 - Q_{\leq  \lambda_1 \lambda_2})v_{\lambda_2}}{L^1_t L^2_x} + \lambda_1\norm{ u_{\lambda_1}}{L^{\infty}} \lambda_2^3 \norm{( 1 - Q_{\leq  \lambda_1 \lambda_2})v_{\lambda_2}}{L^1_t L^2_x}\\
		&\approx \norm{u_{\lambda_1} L ( 1 - Q_{\leq  \lambda_1 \lambda_2})v_{\lambda_2}}{L^1_t L^2_x}\\
		&~~~ + \norm{ u_{\lambda_1}}{L^{\infty}} \lambda_1 \lambda_2^3   ( \lambda_1 \lambda_2)^{-1}\norm{( 1 - Q_{\leq  \lambda_1 \lambda_2})v_{\lambda_2}}{ ( \lambda_1 \lambda_2)^{-1}L^1_t L^2_x}
	\end{align*}
	where we note $ \lambda_1 \ll \lambda_2$. We now proceed by Lemma \ref{multiplierr} $(b)$ (for $ \mu = \lambda_1 \lambda_2$)
	\begin{align*}
		\lambda_2^{-2}  \|L(u_{\lambda_1} &( 1 - Q_{ \leq  \lambda_1 \lambda_2})v_{\lambda_2})_{\lambda_2}\|_{L^1_t L^2_x}\\
		&\lesssim \lambda_1^{\f{d}{2}} \norm{ u_{\lambda_1}}{Z_{\lambda_1}} ( \lambda_2^{-2} \norm{Lv_{\lambda_2}}{L^1_t L^2_x} +  \norm{ v_{\lambda_2}}{Y_{\lambda_2}}),
	\end{align*}
	which gives the claim. The proof part $(b)$ follows similarly, in fact easier, since we can directly place $ (u_{\lambda_2} v_{\lambda_2})_{\lambda_1} \in X^{\f12, 1}_{\lambda_1}$ by estimating
	\begin{align}\label{hilfs2}
		\norm{u_{\lambda_2}v_{\lambda_2}}{L^2_{t,x}} \lesssim \lambda_2^{\f{d}{2}-1} \norm{u_{\lambda_2}}{Z_{\lambda_2}} \norm{v_{\lambda_2}}{Z_{\lambda_2}}.
	\end{align}
	Then, from $ X_{\lambda_1}^{\f12, 1} \subset Z_{\lambda_1}$, \eqref{hilfs2} gives
	\begin{align*}
		\norm{ (u_{\lambda_2}v_{\lambda_2})_{\lambda_1}}{Z_{\lambda_1}} &\lesssim 
		\sum_{\mu \leq 4  \lambda_1^2} \left(\f{\mu}{\lambda_1^2}\right)^{\f12}  \lambda_2  \norm{u_{\lambda_2}v_{\lambda_2}}{L^2_{t,x}}\\
		& \lesssim \lambda_2 ( \lambda_2^{\f{d}{2}-1} \norm{u_{\lambda_2}}{Z_{\lambda_2}} \norm{v_{\lambda_2}}{Z_{\lambda_2}}).
	\end{align*}
	For \eqref{hilfs2}, we write $ u_{\lambda_2} v_{\lambda_2} =  \sum_{e \in \mathcal{M}} u_{\lambda_2}^e v_{\lambda_2}$ where $ u_{\lambda_2}^e \in S_{\lambda_2}^e$. Hence
	\begin{align*}
		\norm{u_{\lambda_2}^e v_{\lambda_2}}{L^2_{t,x}} &\leq \norm{u^e_{\lambda_2}}{L^{\infty}_e L^{2}_{t, e^{\perp}} } \norm{v_{\lambda_2}}{L^{2}_e L^{\infty}_{t, e^{\perp}}}\\
		& \leq  \lambda_2^{-\f12} \norm{u_{\lambda_2}^e}{\lambda_2^{-\f12}L_e^{\infty} L^2_{t, e^{\perp}}} \lambda_2^{\f{d-1}{2}} \norm{v_{\lambda_2}}{\lambda_2^{\f{d-1}{2}}\bigcap_{\tilde{e}}L_{\tilde{e}}^{2} L^{\infty}_{t, \tilde{e}^{\perp}}}.
	\end{align*}
	Summing over $ e \in \mathcal{M}$, we infer the claim.
	
\end{proof}
From Lemma \ref{algebra}, we obtain \eqref{bilinear-embeddings1} as outlined above by summation according to the definiton of $Z^{\f{d}{2}} $ and $ W^{\f{d}{2}}$. Note that the estimates for the remaining frequency interactions in \eqref{yes} and \eqref{yyes} follow the same arguments provided in Lemma \ref{algebra}.
\absatz
Similarly, for the embedding \eqref{bilinear-embeddings2} we prove the subsequent estimates.
\begin{Lemma}\label{algebra-nonlinearity}
	\begin{align}\label{wfirst}
		&\norm{u_{\lambda_2}v_{\lambda_1}}{W^{\f{d}{2}}} \lesssim \lambda_1^{\f{d}{2}}\lambda_2^{\f{d}{2}} \norm{u_{\lambda_2}}{Z_{\lambda_2}} \norm{v_{\lambda_1}}{W_{\lambda_1}},~~ \lambda_1 \leq \lambda_2\\[3pt]
		\label{wsecond}
		&\norm{u_{\lambda_2}v_{\lambda_1}}{W_{\lambda_2}} \lesssim \lambda_1^{\f{d}{2}} \norm{u_{\lambda_2}}{W_{\lambda_2}} \norm{v_{\lambda_1}}{Z_{\lambda_1}},~~ \lambda_1 \ll \lambda_2
	\end{align}
\end{Lemma}
~~\\
\begin{proof}[Proof]
	We first estimate by Sobolev embedding
	\begin{align*}
		\lambda_2^{-2}\norm{u_{\lambda_2}v_{\lambda_1}}{L^1_t L^2_x}&\lesssim \lambda_2^{-2}\norm{u_{\lambda_2}}{L^2_t L^{\f{2d}{d-2}}_x}  \norm{v_{\lambda_1}}{L^2_t L^{d}_x}\\
		& \lesssim  \lambda_2^{-2}\norm{u_{\lambda_2}}{L^2_t L^{\f{2d}{d-2}}_x} \lambda_1^{\f{d-2}{2}} \norm{v_{\lambda_1}}{L^2_{t,x}}\\
		& \lesssim \norm{u_{\lambda_2}}{S_{\lambda_2}} \lambda_1^{\f{d}{2}-3}\norm{v_{\lambda_1}}{L^2_{t,x}}\\
		& \lesssim \norm{u_{\lambda_2}}{Z_{\lambda_2}} \lambda_1^{\f{d}{2}}\norm{v_{\lambda_1}}{W_{\lambda_1}}
	\end{align*}
	where we used Lemma \ref{emb-Lemma} for $ W_{\lambda_1} \subset \lambda_1^{3}L^2_{t,x}$.
	Thus from
	\begin{align*}
		\norm{u_{\lambda_2}^e v_{\lambda_1}}{W^{\f{d}{2}}} \lesssim  \lambda_2^{\f{d-4}{2}}\norm{u_{\lambda_2}v_{\lambda_1}}{L^1_t L^2_{x}} \lesssim  \lambda_2^{\f{d}{2}} \lambda^{\f{d}{2}}_1 \norm{u_{\lambda_2}}{Z_{\lambda_2}} \norm{v_{\lambda_1}}{W_{\lambda_1}},
	\end{align*}
	we obtain the claim \eqref{wfirst}. Estimate \eqref{wsecond} is implied by
	\begin{align}
		\lambda_1^{- \f{d}{2}}Z_{\lambda_1} \cdot L_t^1 L_{x}^2 &\subset  L_t^1 L_{x}^2,\\ \label{secccond}
		\lambda_1^{- \f{d}{2}}Z_{\lambda_1} \cdot X^{- \f{1}{2},1}_{\lambda_2} &\subset  W_{\lambda_2}\lambda_2^{-2},
	\end{align} 
	where the first embedding follows from $ Z_{\lambda_1} \subset \lambda_1^{ \f{d}{2}} L^{\infty}_{t,x}$. For \eqref{secccond}, we note that  since we restrict to $ \lambda_1 \ll \lambda_2 $, we only consider  $ \lambda_1 \leq \f{\lambda_2}{C} $   for a large, fixed  constant $ C > 0 $. We thus decompose
	$$ u_{\lambda_2} = Q_{\leq C^2 \lambda_1^2}  u_{\lambda_2} +  (1 - Q_{\leq C^2 \lambda_1^2}) u_{\lambda_2}.$$
	In particular, each dyadic piece  $Q_{\mu} u_{\lambda_2}$ in $ (1 - Q_{\leq C^2 \lambda_1^2})X^{- \f12, 1}_{\lambda_2} $ satisfies
	$\lambda_1^2 \ll \mu \leq \lambda_2^2.$ We then estimate (note that we use \eqref{iden-mod})
	\begin{align*}
		\norm{v_{\lambda_1}(1 - Q_{\leq C^2 \lambda_1^2}) u_{\lambda_2} }{X^{-\f12,1}_{\lambda_2}} &\sim  \sum_{C^2 \lambda_1^2 \leq \mu \leq 4 \lambda_2^2}\mu^{-\f12} \norm{v_{\lambda_1} Q_{\mu} u_{\lambda_2}}{L^2_{t, x}}\\
		&\lesssim   \sum_{C^2 \lambda_1^2 \leq \mu \leq 4 \lambda_2^2}\mu^{-\f12} \norm{v_{\lambda_1}}{L^{\infty}_{t,x}}\norm{ Q_{\mu} u_{\lambda_2}}{L^2_{t, x}}\\
		& \lesssim \lambda_1^{\f{d}{2}} \norm{v_{\lambda_1}}{Z_{\lambda_1}}\norm{  u_{\lambda_2} }{X^{-\f12,1}_{\lambda_2}}.
	\end{align*}
	Further 
	\begin{align*}
		\lambda_2^2\norm{v_{\lambda_1} Q_{\leq C^2 \lambda_1^2}u_{\lambda_2}}{W_{\lambda_2}} &\lesssim  \norm{ v_{\lambda_1}Q_{\leq C^2 \lambda_1^2}u_{\lambda_2} }{L^1_t L^{2}_{x}}\\
		&\lesssim \norm{ v_{\lambda_1} }{L^2_t L^{\infty}_{x}}  \norm{Q_{\leq C^2 \lambda_1^2}u_{\lambda_2}}{L^2_{t,x}}\\
		&\lesssim  \lambda_1^{\f{d}{2}} \norm{ v_{\lambda_1} }{\lambda_1^{\f{d-2}{2}}L^2_t L^{\infty}_{x}}\sum_{\mu \leq C^2 \lambda_1^2} \mu^{- \f12} \norm{Q_{\mu}u_{\lambda_2}}{L^2_{t,x}}\\
		&\lesssim \lambda_1^{\f{d}{2}} \norm{v_{\lambda_1}}{Z_{\lambda_1}} \norm{u_{\lambda_2}}{X^{- \f12,1}_{\lambda_2}},
	\end{align*}
	which follows from $ Z_{\lambda_1} \subset S_{\lambda_1}$.
\end{proof}
~~\\
We now infer \eqref{bilinear-embeddings1} and \eqref{bilinear-embeddings2} by the summation argument provided in the beginning of the section.
\subsection{Higher regularity}
The percisteny of higher  regularity of the $  \dot{B}^{2,1}_{\f{d}{2}}\times \dot{B}^{2,1}_{\f{d}{2}-2} $ solution as stated in Theorem \ref{main1}  follows as in \cite{tataru1} and \cite{bejenaru1} from  \eqref{higher-regularity1} and \eqref{higher-regularity2}. We will briefly outline how to employ these estimates for the proof of Theorem \ref{main1} and Corollary \ref{main2} in the next Section \ref{sec:proof}.
\absatz
For \eqref{higher-regularity1}, we rely again on Lemma \ref{algebra}  and  the decomposition 
$$  uv = \sum_{\lambda_1 \ll \lambda_2} u_{\lambda_2}v_{\lambda_1} + \sum_{\lambda_2 \ll \lambda_1} u_{\lambda_2}v_{\lambda_1}  + \sum_{\lambda_1 \sim \lambda_2} u_{\lambda_2}v_{\lambda_1},$$
from the beginning of the Section \ref{sec:multi}. However, we now sum as follows 
\begin{align*}
	\bigg(\sum_{\lambda} \lambda^{2s} \norm{(uv)_{\lambda}}{Z_{\lambda}}^2 \bigg)^{\f12} \lesssim& ~ \sum_{\lambda_1} \bigg( \sum_{\lambda} \lambda^{2s} \bigg \| \bigg(\sum_{\lambda_1 \ll \lambda_2} u_{\lambda_2} v_{\lambda_1}\bigg)_{\lambda}\bigg \|_{Z_{\lambda}}^2 \bigg)^{\f12}\\
	& +  \sum_{\lambda_2} \bigg( \sum_{\lambda} \lambda^{2s} \bigg \|  \bigg(\sum_{\lambda_2 \ll \lambda_1} u_{\lambda_2} v_{\lambda_1}\bigg)_{\lambda} \bigg \|_{Z_{\lambda}}^2 \bigg)^{\f12} + \sum_{\lambda_1 \sim  \lambda_2} \norm{u_{\lambda_2} v_{\lambda_1}}{Z^s}.
\end{align*}
Hence, we need to estimate the three terms
\begin{align*}
	\sum_{\lambda_1} \bigg( \sum_{\lambda_1 \ll \lambda } \lambda^{2s}  \| (u_{\lambda} v_{\lambda_1})_{\lambda}\|_{Z_{\lambda}}^2 \bigg)^{\f12},~~~ \sum_{\lambda_2} \bigg( \sum_{\lambda_2 \ll \lambda } \lambda^{2s}  \| (u_{\lambda_2} v_{\lambda})_{\lambda}\|_{Z_{\lambda}}^2 \bigg)^{\f12},~~~\sum_{\lambda_1 \sim  \lambda_2} \norm{u_{\lambda_2} v_{\lambda_1}}{Z^s},
\end{align*}
where for $ s > \f{d}{2}$, the latter sum is treated by Lemma \ref{algebra} $(b)$ similar as before via (note that we identify $\lambda_1$ and $ \lambda_2$ for simplicity)
$$ \sum_{\lambda_2} \sum_{\lambda \lesssim \lambda_2} \lambda^s \norm{( u_{\lambda_2} v_{\lambda_2})_{\lambda}}{Z_{\lambda}} \lesssim \sum_{\lambda_2} ( \lambda_2^{2s} \norm{u_{\lambda_2}}{Z_{\lambda_2}}^2)^{\f12}  \lambda_2^{\f{d}{2}} \norm{v_{\lambda_2}}{Z_{\lambda_2}} \lesssim \norm{u}{Z^s} \norm{v}{Z^{\f{d}{2}}}. $$
The LHS of this inequality now bounds the $ l^2(\Z)$ norm (wrt $ \lambda$) and for the first two sums above we directly estimate the squares via Lemma \ref{algebra} $(a)$. For \eqref{higher-regularity2}, we sum in the same way and use the following dyadic estimates
\begin{align*}
	&\sum_{\lambda_1 \lesssim \lambda_2} \lambda_1^s \norm{(u_{\lambda_2}v_{\lambda_2})_{\lambda_1}}{W_{\lambda_1}} \lesssim \lambda_2^{s + \f{d}{2}} \norm{u_{\lambda_2}}{Z_{\lambda_2}} \norm{v_{\lambda_2}}{W_{\lambda_2}},\\[5pt]
	&\norm{(u_{\lambda_2}v_{\lambda_1})_{\lambda_2}}{W_{\lambda_2}} \lesssim \lambda_1^{\f{d}{2}} \norm{u_{\lambda_2}}{W_{\lambda_2}} \norm{v_{\lambda_1}}{Z_{\lambda_1}},~~ \lambda_1 \ll \lambda_2\\[5pt]
	&\norm{(u_{\lambda_2}v_{\lambda_1})_{\lambda_2}}{W_{\lambda_2}} \lesssim \lambda_1^{\f{d}{2}} \norm{u_{\lambda_2}}{Z_{\lambda_2}} \norm{v_{\lambda_1}}{W_{\lambda_1}},~~ \lambda_1 \ll \lambda_2
\end{align*}
which are the same as in (or follow from) Lemma \ref{algebra-nonlinearity}.
	\section{Proof of the main theorem}\label{sec:proof}
	The proof of Theorem \ref{main1} follows straight forward perturbatively by convergence of 
	\begin{align}\label{seq}
		u_{k+1} = Su[0] + V( \mathcal{Q}(u_k)),~~ k \geq 0,~ u_0(t,x) = 0
	\end{align}
	in the space $ Z^{\f{d}{2}}$, where $ Su[0]  = S (u_0,u_1)$ solves \eqref{equation-linear} with $ F = 0 $ and $ V F $ solves \eqref{equation-linear} with vanishing initial data. To be more precise, we combine Lemma \ref{Linea-propo}, i.e.
	\begin{align}
		&\norm{u_{k+1}}{Z^{\f{d}{2}}} \lesssim \norm{u_0}{\dot{B}^{2,1}_{\f{d}{2}}} + \norm{u_1}{\dot{B}^{2,1}_{\f{d}{2}-2}}
		+ \norm{\mathcal{Q}(u_k)}{W^{\f{d}{2}}}\\[3pt]
		&\sup_{t \in \R} \big(\norm{u(t)}{\dot{B}^{2,1}_{\f{d}{2}}} + \norm{\partial_tu(t)}{\dot{B}^{2,1}_{\f{d}{2}-2}}\big) \lesssim \norm{u}{Z^{\f{d}{2}}},
	\end{align}
	with the Lipschitz estimate 
	\begin{align}\label{Lip}
		\norm{\mathcal{Q}(u_{k}) - \mathcal{Q}(v_{k})}{W^{\frac{d}{2}}} \lesssim C(\norm{u_{k}}{Z^{\f{d}{2}}},\norm{v_{k}}{Z^{\f{d}{2}}} )\norm{u_k - v_k }{Z^{\f{d}{2}}}.
	\end{align}
	This is a direct consequence of  \eqref{bilinear-embeddings1},~ \eqref{bilinear-embeddings2} combined with the identity \eqref{null-structure}, i.e.
$$	\mathcal{Q}(u) = \f12Q_u( L( u \cdot u) -  u \cdot Lu - Lu \cdot u)$$
	 and Lemma \ref{Linea-propo} provided  $\mathcal{Q} $ is analytic (at $x_0 = 0$), $ \delta > 0 $ is small enough and \eqref{small} holds. This is necessary to expand the coefficients of $ \mathcal{Q}$,  which then converge uniformely near $ x_0 = 0$ hence in $ Z^{\frac{d}{2}}$. Especially, for $ \delta > 0 $ sufficiently small \eqref{seq} converges to a solution of \eqref{general2} in the $ \delta$-ball of $ Z^{\f{d}{2}}$ centered at $u = 0$. For higher regularity, we proceed as in \cite{tataru1} and construct a solution for \eqref{general2} in the space  $ Z^{\f{d}{2}}\cap Z^s $ with norm
	$$ \norm{u}{Z^{\f{d}{2}}\cap Z^s} = \f{1}{M}\norm{u}{ Z^s} +\f{1}{\tilde{\delta}} \norm{u}{Z^{\f{d}{2}}},$$
	where aditionally $ (u_0, u_1) \in \dot{H}^s(\R^d) \times \dot{H}^{s-2} (\R^d) $ for some $ s > \f{d}{2}$.
	Then, provided $  \tilde{\delta } < \delta $, the estimates \eqref{higher-regularity1} and \eqref{higher-regularity2} imply
	a Lipschitz estimate similar to \eqref{Lip} for  $W^s $ on the LHS and $ Z^{\f{n}{2}}\cap Z^s $ on the RHS. More precisely, there holds
	\begin{align*}
		\norm{\mathcal{Q}(u) - \mathcal{Q}(v)}{W^s} \lesssim \norm{u - v}{Z^s} (\norm{u}{Z^{\f{d}{2}}} + \norm{v}{Z^{\f{d}{2}}}) + \norm{u - v}{Z^{\f{d}{2}}}(\norm{u}{Z^s} + \norm{v}{Z^s}).
	\end{align*}
	Especially, with the corresponding linear estimates as above, \eqref{seq} converges in the unit ball of $  Z^{\f{d}{2}}\cap Z^s $ (in the above norm), where we take 
	$$ M \sim  \norm{u_0}{\dot{H}^s(\R^d)} + \norm{u_1}{\dot{H}^{s-2}(\R^d)}.$$
	In order to obtain the Lipschitz estimate and the fact that the fixed point operator maps the unit ball into itself, we note that from \eqref{higher-regularity1} (combined with \eqref{bilinear-embeddings1}) there holds by induction over $ k \in \N $ for $ u,v \in Z^{\f{d}{2}}\cap Z^s$
	\begin{align*}
		\norm{u^k}{Z^s} &\lesssim k \norm{u}{Z^{\f{d}{2}}}^{k-1} \norm{u}{Z^s}, \\[3pt]
		\norm{(u-v)u^{k-1}}{Z^s} &\lesssim \norm{u-v}{Z^s} \norm{u}{Z^{\f{d}{2}}}^{k-1} + (k-1)\norm{u}{Z^{\f{d}{2}}}^{k-2} \norm{u-v}{Z^{\f{d}{2}}} \norm{u}{Z^s}.
	\end{align*}
	In particular, the smallness assumption is only necessary in $ Z^{\f{d}{2}} $ in order to estimate the series expansion of $\mathcal{Q}(u),~\mathcal{Q}(v)$ in \eqref{general}. Thus from \eqref{higher-regularity2} and \eqref{bilinear-embeddings2} we infer
	\begin{align*}
		&\norm{V( \mathcal{Q}(u))}{Z^{s}} \lesssim \norm{u}{Z^s}\norm{u}{Z^{\f{d}{2}}} \lesssim \tilde{\delta} \norm{u}{Z^{s}},\\[3pt]
		&\norm{V( \mathcal{Q}(u))}{Z^{\f{d}{2}}} \lesssim \norm{u^2}{Z^{\f{d}{2}}} \lesssim \tilde{\delta} \norm{u}{Z^{\f{d}{2}}}.
	\end{align*}
	Similarly, for the difference $ u-v $ of $ u,v \in Z^s \cap Z^{\f{d}{2}}$, we infer the Lipschitz estimate.
	Since $ \tilde{\delta} < \delta$, any such solution also lies in the $\delta$-ball in $Z^{\f{d}{2}} $ and thus coincides with the solution in this space.
	\\[10pt]
	The second problem \eqref{bihom} is treated similarly, i.e. we expand 
	\begin{align}\label{sums}
	\Pi(x) = \sum_{k = 0}^{\infty}  \frac{1}{k !} d^k\Pi(x)_{|_{x = 0}} ( x^{k}),
	\end{align}
	where $ d^k \Pi(x)$ are $ k $-tensors with the notation for $l = 1, \dots ,L$.  Especially we have for any $ v \in \R^L $ 
	\begin{align}\label{ja-for-pi}
		d\Pi_x(v) =& \sum_{k = 1}^{\infty}  \frac{1}{(k-1) !} \sum_{l=1}^Ld^{k-1}\partial_{x_l}\Pi(x)_{|_{x = 0}} ( x^{k-1}) v_l\\[3pt]\notag
		=& \sum_{k = 1}^{\infty}  \frac{1}{k !} \sum_{l=1}^Ld^{k-1}\partial_{x_l}\Pi(x)_{|_{x = 0}} ( x^{k-1}) k v_l.
	\end{align}
	Since now consider $	\mathcal{N}(u) =L( \Pi(u)) -  d \Pi_u(Lu) $
we  note that by continuity of $ L : Z^{\f{d}{2}} \to W^{\f{d}{2}}$, i.e. 
$$ \norm{Lv}{W^{\f{d}{2}}} \lesssim \norm{v}{Z^{\f{d}{2}}} \lesssim \delta,$$
and by convergence of the series in  $B^{Z^{\f{d}{2}}}(0, \delta )$, we justify to pull $L$ into the series expansion and all terms in the series expression of $ \mathcal{N}(u)$ are at least quadratic. More precisely 
$$ \mathcal{N}(u) = \sum_{k \geq 2} \frac{1}{k !}  (d^k\Pi(x))_{|_{x = 0}} ( L (u^k) - k u^{k-1} Lu ), $$
converges absolutely in $W^{\f{d}{2}}$ if $ u \in  B^{Z^{\f{d}{2}}}(0, \delta )$ and $\delta > 0 $ is small enough.  Similarly for the difference we have
$$ \mathcal{N}(u) - \mathcal{N}(v) = \sum_{k \geq 2} \sum_{l = 0}^{k-1}\frac{1}{k !}  (d^k\Pi(x))_{|_{x = 0}} ( L (v^l w u^{k-1-l}) - k v^{l} w u^{k-2 -l} Lu - k v^{k-1} Lw ),$$
where for the middle term, we only sum $l = 0, \dots, k-2$. In this notation e.g. $ (d^k\Pi(x))_{|_{x = 0}} (v^l w u^{k-l-1})$ captures all terms of the form 
\begin{align*}
	\sum_{\substack{l_1 + \dots l_{m} = l\\ l_{m+2} \dots + l_L = k-1-l}}C_{l_1,\dots, l_{L}} (\partial_{x_{i_1}}^{l_1} \cdots \partial_{x_{i_L}}^{l_L}\Pi(0)) v_{i_1}^{l_1} \cdots v_{i_m}^{l_{m}}w_{i_{m+1}} u_{i_{m+2}}^{l_{m+2}}\cdots u_{i_L}^{l_{L}},~~~i_j \in \{1, \dots,L\}.
\end{align*} 
Then the argument above applies and we now want to construct a global solution of \eqref{jap}, which reads as
	\begin{align*}
		\partial_t^2u + \Delta^2u &= dP_u(u_t,u_t) + dP_u(\Delta u,\Delta u) + 4 dP_u( \nabla u, \nabla \Delta u) + 2 dP_u(\nabla^2u, \nabla^2 u)\\ \nonumber
		&\quad+2 d^2P_u(\nabla u, \nabla u, \Delta u)  + 4 d^2P_u( \nabla u, \nabla u, \nabla^2 u )\\ \nonumber
		&\quad+ d^3P_u(\nabla u, \nabla u, \nabla u, \nabla u),
	\end{align*}
	where 
	\begin{align*}
		d^2P_u( \nabla u, \nabla u, \nabla^2 u ) &= d^2P_u(\partial^i u, \partial_ju, \partial_i \partial^j u),\\[3pt]
		d^3P_u(\nabla u, \nabla u, \nabla u, \nabla u) &= d^3P_u(\partial_i u, \partial^i u, \partial_ju, \partial^j u),
	\end{align*}
	and $ dP_u,~d^2P_u,~d^3P_u$ are derivatives of the orthogonal tangent projector $ P_p : \R^L \to T_pN $ for $ p \in N $.
	We extend this equation via $\Pi$ ($d\Pi_u = P_u$ for $u \in N$) to functions that only map to the neighborhood $ \mathcal{V}_{\varepsilon}(N)$. By direct calculation or comparison to \eqref{jap}, this can be verified for  \eqref{bihom} and hence we solve
	$$ Lv = L(\Pi(v+p)) - d\Pi_{v+p}(Lv),$$
	for $v = u - p$ where $  p: = \lim_{|x|  \to \infty} u_0(x) $ via the $Z^{\f{d}{2}}\cap Z^s$-limit of
	\begin{align}\label{seq2}
		v_{k+1} = Sv[0] + V( \mathcal{N}(v_k)),~~ k \geq 0,~ v_0(t,x) = 0,~~ \mathcal{N}(v) = L(\Pi(v)) - d \Pi_v(Lv).
	\end{align}
	In particular, the smoothness of the solution follows from the persistence of higher regularity in the fixed point argument from above. Since for $ \delta > 0 $ small enough, we obtain (note that in  $\dot{B}^{2,1}_{\f{d}{2}} $ we have $ C_0$ data)
	\begin{align}\label{distas}
		\sup_{t \in \R} \dist(u, N) \leq  \norm{u-p}{L^{\infty}_{t,x}} \lesssim  \norm{v}{L^{\infty}_t B^{2,1}_{\f{d}{2}}} \lesssim  \norm{v}{Z^{\f{n}{2}}} \lesssim \delta,
	\end{align}
	the map $ \Pi $ and thus \eqref{bihom} is welldefined in a $B(0, C\delta)$ ball in $Z^{\f{d}{2}}$. The only thing left to show is that $ u(t) \in N$ for $t \in \R$, such that in particular, \eqref{bihom} implies  \eqref{jap}.
	\absatz
	If $ v = u-p \in B(0, C \delta) \subset Z^{\f{d}{2}}$, then $ \Pi(u) - p = \Pi(v + p) - p \in Z^{\f{d}{2}}$ and $ \Pi(u) - u = \Pi(v+p) - p - v \in Z^{\f{d}{2}}$ with
	\begin{align*}
		\norm{\Pi(u) - u}{Z^{\f{d}{2}}}  + \norm{\Pi(u) - p}{Z^{\f{d}{2}}} \lesssim \norm{v}{Z^{\f{d}{2}}},
	\end{align*} 
	provided $ \delta > 0 $ is small. We now have
	\begin{align}
		L( u - \Pi(u)) = L(v - \Pi(v+p)) = - d\Pi_{v + p}(Lv)  = -d\Pi_{v + p}(\mathcal{N}(v)).
	\end{align}
	 Since $ \Pi(u) \in N $, we have $ \mathcal{N}( \Pi(u) - p) \perp T_{\Pi(u)}N $ and from $ Im(d \Pi_u) \subset T_{\Pi(u)} N $, $u = v + p$, we obtain 
	$$ d\Pi_{v + p}(\mathcal{N}( \Pi(u)-p)) = 0.$$ 
	At this point, however, we mention, that we cancel the linear part in  series expansions for $ L(\Pi(v + p))$ in $\mathcal{N}(v)$ and of $ L(\Pi(u) - p) = L(\Pi(v + p) - p)$ in $\mathcal{N}(\Pi(u) - p) $ at $ v = 0$ (Thus the constant part of $d\Pi$ vansihes in $\mathcal{N}$). We then obtain
	\begin{align}
		L( u - \Pi(u)) = -d\Pi_{v + p}\big(& (d\Pi_{(\Pi(u) -p) +p} - d\Pi_{v + p}) L( \Pi(u) - p)\\[2pt] \nonumber
		&+ d\Pi_{v +p}( L( \Pi(u) - p - v))\\[2pt] \nonumber
		&+ L( \Pi(v + p) - \Pi((\Pi(u) - p) + p))\big).
	\end{align}
	Note that we don't want to use $ \Pi^2= \Pi$, since technically we want the identity for the series expressions for $\Pi,~ d \Pi $ with missing linear parts. Especially, all terms appearing on the RHS are at least quadratic.\\[4pt]
	This implies (note that $u(0) = \Pi u(0),~ u_t(0) = \partial_t(\Pi u)(0)$ by assumption)
	\begin{align*}
		\norm{u - \Pi(u)}{Z^{\f{d}{2}}} &\lesssim~( 1 + \norm{v}{Z^{\f{d}{2}}})\norm{ (d\Pi_{(\Pi(u) -p) +p} - d\Pi_{v + p}) L( \Pi(u) - p)}{W^{\f{d}{2}}}\\
		&~~~~+ ( 1 + \norm{v}{Z^{\f{d}{2}}})\norm{d\Pi_{v +p}( L( \Pi(u) - p - v)) }{W^{\f{d}{2}}}\\
		&~~~~+ ( 1 + \norm{v}{Z^{\f{d}{2}}})\norm{ L( \Pi(v + p) - \Pi((\Pi(u) - p) + p))}{W^{\f{d}{2}}}\\[4pt]
		&\lesssim ( 1 + \norm{v}{Z^{\f{d}{2}}})\norm{u - \Pi(u)}{Z^{\f{d}{2}}} \norm{L(\Pi(u) - p)}{W^{\f{d}{2}}}\\
		&~~~+ ( 1 + \norm{v}{Z^{\f{d}{2}}})\norm{v}{Z^{\f{d}{2}}} \norm{L(\Pi(u) - p - v)}{W^{\f{d}{2}}}\\
		&~~~+ ( 1 + \norm{v}{Z^{\f{d}{2}}})(\norm{v}{Z^{\f{d}{2}}} + \norm{\Pi(u) - p}{Z^{\f{d}{2}}}) \norm{u - \Pi(u)}{Z^{\f{d}{2}}}\\[4pt]
		&\lesssim ( 1 + \norm{v}{Z^{\f{d}{2}}}) \norm{v}{Z^{\f{d}{2}}} \norm{u - \Pi(u)}{Z^{\f{d}{2}}}.
	\end{align*}
	In particular, if $ \norm{v}{Z^{\f{d}{2}}}  \leq \delta$ is sufficiently small, we have $ u = \Pi(u) \in N $.

	\appendix
	
	\section{Local Smoothing  \& lateral Strichartz inequalities}\label{appendix}
	~~\\
	In this section, we recall  the \emph{local smoothing} effect (i.e. lateral Strichartz estimates with localized data) and a \emph{maximal function estimate} for the linear Cauchy problem
	\begin{align}\label{Schro}
		\begin{cases}
			i \partial_t u (t,x) \pm \Delta u(t,x) = f(t,x) & (t,x) \in \R \times \R^d\\[2pt]
			u(0,x) = u_0(x)& x \in \R^d
		\end{cases}
	\end{align}
	in the  lateral space $L^p_e L^q_{t, e^{\perp}} $ for $ e \in \mathbb{S}^{d-1}$  with norm
	\begin{align}\label{resolution}
		\norm{f}{L^p_e L^q_{t, e^{\perp}}}^p =  \int_{-\infty}^{\infty} \left( \int_{[e]^{\perp}} \int_{- \infty}^{\infty}  | f( t, r e + x)|^q d t~d x\right)^{\f{p}{q}}d r.
	\end{align}
	The norm \eqref{resolution} was used by Kenig, Ponce, Vega, see e.g. \cite{KPV}, in order to establish local smoothing estimates for nonlinear Schr\"odinger equations.\\[3pt]
	The estimates for $L^p_e L^q_{t, e^{\perp}},~ L^1_e L^2_{t, e^{\perp}},~~L^2_e L^{\infty}_{t, e^{\perp}} $ in Corollary \ref{Corol-Strichartz} and Lemma \ref{maxfunc} below are substantial in the wellposedness theory of Schr\"odinger maps and were proven by Ionescu, Kenig in \cite{ionescu-kenig1},\cite{ionescu-kenig2} (see also the work of Bejenaru in \cite{bejenaru1} and Bejenaru, Ionescu, Kenig in \cite{bejenaru2}).\\[3pt]
	Similar ideas (however more involved due to the absence of the $L^2_e L^{\infty}_{t, e^{\perp}}$ estimate in $ d = 2$) have been used by Bejenaru, Ionescu, Kenig and Tataru in \cite{bejenaru3} for  global Schr\"odinger maps into $ \mathbb{S}^2$ in dimension $ d \geq 2$ with small initial data in $H^{\f{d}{2}}$.
	\absatz
	Here we follow Bejenaru's calculation in \cite{bejenaru1}, which recovers the smoothing effect for \eqref{Schro} provided the data $ u_0, f$ is sufficiently localized in the sets
	\begin{align*}
		A_e &= \big\{ \xi~|~ \xi \cdot e \geq \f{|\xi|}{\sqrt{2}}\big\},\\
		B^{\pm}_e  &= \left \{ (\tau,\xi)~|~ | \pm \tau - \xi^2| \leq \f{|\tau| + \xi^2}{10},~~\xi \in A_e \right \}\\[4pt]
		A_{\lambda} &= ~\{ (\tau, \xi) ~|~ \lambda/2 \leq (\tau^2 + |\xi|^4)^{\f14} \leq 2 \lambda \},
	\end{align*}
	as defined in Section \ref{sec:Linear-est-func}. Especially for $ (\tau, \xi) \in B_e^{\pm} \cap A_{\lambda}$, there holds
	\begin{align}\label{facts2}
		\pm \tau - \xi_{e^{\perp}}^2 \geq 0,~~~~ \xi_e \sim \lambda,~~~~\xi_e + \sqrt{ \pm \tau - \xi_{e^{\perp}}^2} \sim \lambda.
	\end{align} 
	\begin{Rem}
		We note that our definition of $B_e^{\pm}$ slightly differs from \cite{bejenaru1}.
	\end{Rem}
	Taking the FT (in $t,x$) of \eqref{Schro}, with $ u$ being localized in  $ B^{\pm}_e$,
	\begin{align}
		\hat{f}(\tau, \xi) = (\tau \mp |\xi|^2) \hat{u}(\tau, \xi) = \pm  \left(\sqrt{\pm \tau - \xi_{e^{\perp}}^2} - \xi_e \right)\left(\sqrt{\pm \tau - \xi_{e^{\perp}}^2} + \xi_e \right)\hat{u}(\tau, \xi).
	\end{align}
	Hence, considering \eqref{facts2}, we proceed by taking the (inv.) FT in the coordinates $ t, x_{e^{\perp}}$,
	\begin{align}
		\pm \mathcal{F}^{-1}(\hat{f}(\xi_e, \tau,& \xi_{e^{\perp}}) \big(\sqrt{\pm \tau - \xi_{e^{\perp}}^2} + \xi_e \big)^{-1} )\\ \nonumber
		& =  \mathcal{F}^{-1}\left(\sqrt{\pm \tau - \xi_{e^{\perp}}^2}\hat{u}(\xi_e, \tau, \xi_{e^{\perp}})\right)- \xi_e \hat{u}(\xi_e, t, x_{e^{\perp}}).
	\end{align}
	Thus, \eqref{Schro} is equivalent to an intial value problem of the following type
	\begin{align}
		\begin{cases} 
			(i \partial_{r} + D^{\pm}_{t, x})v(t,r,x) = f,\label{pseudole} \\[5pt]
			v(t,0,x)= u(t,x),\\
		\end{cases}
	\end{align}
	where $ \widehat{D^{\pm}_{t,x}v}(\tau, \xi) =  \sqrt{ \pm \tau - | \xi|^2}\hat{v}(\tau, \xi)$. Thus (at least formally) the homogeneous solution of \eqref{pseudole} is represented as
	$$ v(t,r,x) = e^{i r D^{\pm}_{t,x}}u(t, x).$$
	In the following, we only consider homogeneous estimates for \eqref{Schro}, wich imply all linear estimates we need in Section\ref{sec:Linear-est-func}. Inhomogeneous bounds for the biharmonic problem \eqref{equation-linear} with $ F \in L^1_e L^2_{t, e^{\perp}}$ can be proven similarly as for the Schr\"odinger equation using the calculation in Section \ref{sec:Linear-est-func}.
	\absatz
	The equation \eqref{pseudole} has the scaling $v_{\lambda}(t,r,x) = v(\lambda^2 t, \lambda r,\lambda x),~ \lambda > 0 $ and we now prove the following Strichartz estimate.
	
	\begin{Lemma}[Strichartz estimate] \label{Strichartz}
		Let $ u \in \mathcal{S}'(\R \times \R^{d-1})$,~ $ f \in \mathcal{S}'(\R \times (\R \times \R^{d-1}))$ have Fourier support in
		$$ \{ \pm \tau \geq \xi^2 \}\cap A_{\lambda}$$
		for some dyadic $ \lambda \in 2^{\Z}$. Then there holds
		\begin{align}\label{Strich1}
			\norm{e^{i r D^{\pm}_{t,x}}u(t,  x)}{L^p_r L^q_{t, x}} &\lesssim \lambda^{ \f{d+1}{2} - \f{1}{p} - \f{d+1}{q} } \norm{ u}{L^2_{t,x}},\\[5pt] \label{Strichartz2}
			\norm{\int_{- \infty}^r e^{i(r-s) D^{\pm}_{t,x} } f(s,t, x)~ds}{L^p_r L^q_{t,x}}  &\lesssim \lambda^{ \f{1}{\tilde{p}'} - \f{1}{p}  +  (d+1)(\f{1}{\tilde{q}'} - \f{1}{q}) -1} \norm{f}{L^{\tilde{p}'}_r L^{\tilde{q}'}_{t,x}},
		\end{align}
		where $(p,q),~(\tilde{p}, \tilde{q})$ are admissible, i.e. $1 \leq p, q \leq \infty,~ (p,q) \neq (2, \infty)  $ if $ d = 2$ and
		\begin{align}
			\f{2}{p}+\f{d}{q} \leq \f{d}{2}.
		\end{align}
	\end{Lemma}
	\begin{proof}
		
		We use a Littlewood-Paley decomposition 
		$$ \widehat{P_{\lambda}(u)}(\tau, \xi) = \varphi((\tau^2 + |\xi|^4)^{\f{1}{4}}\slash \lambda)\hat{u}(\tau, \xi), $$
		and  by scaling of \eqref{pseudole}, we have $ P_{1}(u_{\lambda^{-1}}) = (P_{\lambda}u)_{\lambda^{-1}}$. Thus we reduce the estimate \eqref{Strich1} to
		\begin{align}\label{Strich31}
			\norm{e^{i r D^{\pm}_{t,x}}P_1u(t,  x)}{L^p_r L^q_{t, x}} \leq C_{d,p,q} \norm{ \varphi((\tau^2 + |\xi|^4)^{\f{1}{4}})\hat{u}(\tau, \xi)}{L^2_{\tau,\xi}}.
		\end{align}
		We further have $e^{i r D^{\pm}_{t,x}}P_1u(t,  x) = K *_{t,x} P_1u, $ with kernel
		\begin{align}
			K(t,r,x)  = \int \int e^{i(x,t,r) \cdot (\xi, \tau,\sqrt{ \pm \tau - \xi^2}) } \psi(\tau, \xi) ~d\tau d\xi,
		\end{align}
		where $ \psi \in C^{\infty}_c(\R \times \R^{d-1})$ with $ \psi(\tau, \xi) = 1 $ for $ (\tau, \xi) \in \supp( (\tau, \xi) \mapsto \varphi(( \tau^2 + |\xi|^4)^{\f{1}{4}})))$.
		which is the Fourier transform of a (compactly supported) surface carried measure  on the hypersurface
		$$ S = \left \{ \left ( \xi, \tau, \sqrt{\pm \tau - \xi^2}\right )~|~ \xi \in \R^{d-1}, \tau \in \R, \xi^2 \leq \pm \tau \right \},$$
		and $ S $ has $ d $ non-vanishing principal curvature functions in the relevant coordinate patch. Thus we observe
		\begin{align}
			|K(t,r,x)| \lesssim  ( 1 + |r| )^{- \f{d}{2}},~~ t \in \R,~ x \in \R^{d-1},
		\end{align}
		which gives
		\begin{align}\label{dispersive}
			\norm{e^{i r D^{\pm}_{t,x}}P_1u(t,  x)}{L^{\infty}_{t,x}} \lesssim ( 1 + |r| )^{- \f{d}{2}} \norm{P_1u}{L^1_{t,x}}.
		\end{align}
		Now in the endpoint case $(p,q) = (2, \f{2d}{d-2})$, we need to apply Keel-Tao's argument and otherwise we can use direct interpolation. More precisely, combining \eqref{dispersive} and the fact that $e^{i r D^{\pm}_{t,x}}$ is a group on $L^2_{t,x}$ with a classical $T T^*$ argument and the Christ-Kiselev Lemma, we deduce  \eqref{Strich1} and \eqref{Strichartz2} for $P_1u$.
	\end{proof}
	We remark that Lemma \ref{Strichartz} is valid if $ u, f$ have Fourier support in e.g. in $ A_{\lambda/2} \cup A_{\lambda} \cup A_{2\lambda}$, i.e. as long as frequency is controlled by $\lambda$.
	For a dyadic number $ \lambda \in 2^{\Z}$, we recall the definition of $ A^d_{\lambda} =\{ \xi~|~ \lambda / 2 \leq |\xi| \leq 2\lambda\}$. An immediate consequence of the Strichartz estimate is the following Corollary.
	\begin{Cor}\label{Corol-Strichartz}  Let $ u_0 \in L^2( \R^d)$,~ $e \in \mathbb{S}^{d-1}$,~ $\lambda > 0$ dyadic with $  \supp(\hat{u}_0)\subset A^d_{\lambda} \cap A_e  $.
		Then there holds
		\begin{align}
			\norm{e^{\pm it \Delta} u_0}{L_{e}^p L_{t, e^{\perp}}^q} \leq C \lambda^{ \f{d}{2} - \f{1}{p}- \f{(d+1)}{q}}\norm{u_0}{L^2_{x}},
		\end{align}
		where $ (p,q)$ is an admissible pair. Let $ u \in \mathcal{S}'(\R\times \R^{d})$, ~$e \in \mathbb{S}^{d-1}$,~ $\lambda > 0 $ dyadic such that 
		$$ \supp(\hat{u})\subset \left \{( \tau, \xi_{e^{\perp}} )~|~ (\tau, \xi) \in B^{\pm}_e \cap A_{\lambda} \right \}$$
		Then there holds
		\begin{align}
			\norm{e^{i x_e D^{\pm}_{t,x_{e^{\perp}}}}u(t, x_{e^{\perp}})}{L_t^pL_x^q} \leq C_{d,p,q} \lambda^{\f{d+ 1}{2} - \f{2}{p}- \f{d}{q}}\norm{u(t, x_{e^{\perp}})}{L^2_{t, x_{e^{\perp}}}},
		\end{align}
		where $(p,q)$ is an admissible pair.
		
	\end{Cor}
	
	\begin{proof}
		For the first statement, we identify $ \xi = (\xi \cdot e) e + \xi_{e^{\perp}}  \mapsto ( \xi_e , \xi_{e^{\perp}})$ and proceed as follows. By the change of coordinates $ \sqrt{\pm \tau - \xi_{e^{\perp}}^2} = \xi_e $, we have
		\begin{align*}
			\int_{\R^d} e^{i x \cdot \xi}& e^{\pm i t |\xi|^2} \hat{u}_0 ( \xi) ~d \xi \\
			&= \int_{[e]^{\perp}} \int_{ \{ \pm \tau \geq \xi_{e^{\perp}}^2\}} e^{\pm i t \tau} e^{i x_{e^{\perp}} \xi_{e^{\perp}} } e^{i x_e \sqrt{\pm \tau - \xi_{e^{\perp}}^2}}  \hat{u}_0 ( \sqrt{\pm \tau - \xi_{e^{\perp}}^2}e + \xi_{e^{\perp}}) ~\f{d \tau}{2 \sqrt{\pm\tau - \xi_{e^{\perp}}^2}}~d \xi_{e^{\perp}}.
		\end{align*}
		Now we set
		$$ \hat{u}(\tau, \xi_{e^{\perp}}) = \hat{u}_0 ( \sqrt{\pm \tau - \xi_{e^{\perp}}^2}e + \xi_{e^{\perp}}) \left (2 \sqrt{\pm \tau - \xi_{e^{\perp}}^2}\right)^{-1},~~\text{if}~ \pm \tau \geq | \xi_{e^{\perp}}|^2, $$
		and $\hat{u}(\tau, \xi_{e^{\perp}}) = 0$ elsewhere. By assumption on $ \hat{u}_0 $, we have $ u \in \mathcal
		S'(\R \times \R^{d-1})$ (upon the identification of $[e]^{\perp} = \R^{d-1}$) and for $ (\tau, \xi_{e^{\perp}} ) \in \supp(\hat{u})$ there holds $ \sqrt{\pm \tau - \xi_{e^{\perp}}^2} = \xi_e \sim \lambda$  and $\lambda / 2\leq \sqrt{|\tau| + \xi_{e^{\perp}}^2} \leq 4\lambda $ (since in particular $u_0 $ localizes in $ A_e \cap A_{\lambda}^d$). Thus, we apply Lemma \ref{Strichartz} and conclude
		\begin{align*}
			\norm{e^{\pm it \Delta} u_0(x)}{L_{e}^p L_{t, e^{\perp}}^q} &= \norm{\int_{\R^d} e^{i x \cdot \xi} e^{\pm i t |\xi|^2} \hat{u}_0 ( \xi) ~d \xi}{L_{e}^p L_{t, e^{\perp}}^q}\\
			& = \norm{ \int_{[e]^{\perp}} \int_{ \{ \pm \tau \geq \xi_{e^{\perp}}^2\}} e^{\pm i t \tau} e^{i x_{e^{\perp}} \xi_{e^{\perp}} } e^{i x_e \sqrt{\pm \tau - \xi_{e^{\perp}}^2}}  \hat{u} (\tau, \xi_{e^{\perp}}) d \tau d \xi_{e^{\perp}}}{L_{e}^p L_{t, e^{\perp}}^q}\\
			&\lesssim \lambda^{ \f{d+1}{2} - \f{1}{p} - \f{d+1}{q} }  \norm{\hat{u}_0 ( \sqrt{\tau - \xi_{e^{\perp}}^2}e + \xi_{e^{\perp}}) \left (2 \sqrt{\pm \tau - \xi_{e^{\perp}}^2}\right)^{-1} \chi \{ \pm \tau > \xi_{e^{\perp}}^2 \}}{L^2_{\tau,e^{\perp}}}\\
			& \lesssim\lambda^{ \f{d}{2} - \f{1}{p} - \f{d+1}{q}} \norm{u_0}{L^2_{x}},
		\end{align*}
		where, for the last inequality, we reverse the coordinate change and estimate the Jacobian.
		For the second statement, the estimate follows from Strichartz estimates for the Schr\"odinger group and from the above coordinate transform in the backward direction. To be more precise, we estimate
		\begin{align*}
			\norm{e^{-i x_e D^{\pm}_{t,x_{e^{\perp}}}}u(t, x_{e^{\perp}})}{L_t^pL_x^q} &= \norm{\int_{\R^d} e^{i x \cdot \xi} e^{\pm i t |\xi|^2} \hat{u} ( \pm | \xi|^2, \xi_{e^{\perp}}) ~2 (\xi \cdot e) d \xi}{L_t^pL_x^q}\\
			&\lesssim \lambda^{ \f{d+1}{2} - \f{2}{p} - \f{d}{q} }  \norm{ (\xi \cdot e)^{\f{1}{2}}\hat{u} (\pm | \xi|^2,\xi_{e^{\perp}}) \chi \{ \xi \cdot e \geq 0 \}}{L^2_x}\\
			&\lesssim \lambda^{\f{d+ 1}{2} - \f{2}{p}- \f{d}{q}}\norm{u(t, x_{e^{\perp}})}{L^2_{t, x_{e^{\perp}}}}.
		\end{align*}
	\end{proof}
	\begin{Rem}
		In the case $ \supp(\hat{u}_0) \subset A_{\lambda}^d  $, we obtain  from Corollary \ref{Corol-Strichartz}
		\begin{align}\label{rem}
			\sup_{ e \in \mathcal{M}}\norm{e^{\pm it \Delta}P_e(\nabla) u_0}{L_{e}^p L_{t, e^{\perp}}^q} \leq C \lambda^{ \f{d}{2} - \f{1}{p}- \f{(d+1)}{q}}\norm{u_0}{L^2_{x}},
		\end{align}
		and especially the $L^p_e L^{\infty}_{t, e^{\perp}} $ estimate for $q = \infty,~ pd \geq 4,~ d \geq 3$.
	\end{Rem}
	The next Lemma (from \cite{bejenaru1}) shows that the $P_e(\nabla)$ localization on the LHS of \eqref{rem} is not necessary in the case $ q =\infty $ if $ dp > 4$. 
	\begin{Lemma}\label{maxfunc} Let $ u_0 \in L^2( \R^{d})$ such that $ \supp(\hat{u}_0) \subset A_{\lambda}^d$ for some dyadic  $\lambda \in 2^{\Z}$. Then there holds
		\begin{align}\label{Strich4}
			\sup_{e \in \mathcal{M}}\norm{e^{\pm i t \Delta}u_0}{L^p_e L^{\infty}_{t, e^{\perp}}} \leq C_{d,p} \lambda^{\f{d}{2} - \f{1}{p}  } \norm{u_0}{L^2_{x}},
		\end{align}
		where ~$ 1 \leq p \leq \infty $ and $ dp > 4$. Let $ u \in \mathcal{S}'(\R\times \R^{d})$, such that
		$$ \supp(\hat{u})\subset \{ (\tau, \xi_{\tilde{e}^{\perp}})~|~ (\tau, \xi) \in B_{\tilde{e}}^{\pm} \cap A_{\lambda}\}$$
		for some $\lambda \in 2^{\Z}$ and $ \tilde{e} \in \mathcal{M}$. Then there holds
		\begin{align}\label{Strich5}
			\sup_{e \in \mathcal{M}}\norm{e^{- i r D^{\pm}_{t, x_e}}u(t, x_{e^{\perp}})}{L^p_{r} L^{\infty}_{t, x_{e^{\perp}}}} \leq C_{d,p} \lambda^{ \f{d+1}{2} - \f{1}{p} } \norm{u(t, x_{\tilde{e}^{\perp}})}{L^2_{t,x_{\tilde{e}^{\perp}}}},
		\end{align}
		where $(d,p)$ are as above.
	\end{Lemma}	
	\begin{proof} By scaling we reduce again to the unit frequency $ \lambda = 1 $.
		Then estimate \eqref{Strich4} is a consequence of the $TT^*$ argument for the Schr\"odinger group in the space $L^p_{e}L^{\infty}_{t, e^{\perp}}$ and Young's inequality. As mentioned before, we obtain the decay
		\begin{align*}
			&\left| \int_{\R^d} e^{i x \cdot \xi} e^{\pm i t |\xi|^2} \varphi(|\xi|) d \xi \right|  \lesssim (1 + | x \cdot e|)^{- \f{d}{2}},
		\end{align*}
		which implies (for $ d p > 4$)
		\begin{align*} 
			&\norm{ \int_{\R^d} e^{i x \cdot \xi} e^{\pm i t |\xi|^2} \varphi(|\xi|) d \xi }{L_e^{\f{p}{2}} L^{\infty}_{t, e^{\perp}}} \lesssim 1.
		\end{align*}
		Then by Young's inequality
		\begin{align*} 
			&\norm{ \int e^{\pm i (t-s) \Delta} f(s)~ds }{L_e^p L^{\infty}_{t, e^{\perp}}} \lesssim \norm{ f }{L_e^{p'} L^{1}_{t, e^{\perp}}},
		\end{align*}
		which implies \eqref{Strich4} by $TT^*$. For \eqref{Strich5}, we use	again (note $ \hat{u}$ is localized in $B_e^{\pm}$, thus \eqref{facts2} holds)
		\begin{align}\nonumber
			\int_{\R^d} e^{i x \cdot \xi}& e^{\pm i t |\xi|^2} \varphi(|\xi|) d \xi\\ \nonumber
			&=  \int_{[e]^{\perp}} \int_{\pm \tau \geq \xi_{e^{\perp}}^2} e^{i  x_e \sqrt{\pm \tau - \xi_{e^{\perp}}}}  e^{i ( x_{e^{\perp}}, t) \cdot ( \xi_{e^{\perp}}, \pm \tau)} \varphi_0( \sqrt{\pm \tau}) \f{d \tau}{2 \sqrt{\pm \tau - \xi_{e^{\perp}}}} d \xi_{e^{\perp}} 
		\end{align}
		and thus we obtain \eqref{Strich5} also from $TT^*$ and Young's inequality for $ \exp(- i D^{\pm}_{t, e^{\perp}})$.	
	\end{proof}

	\begin{Rem}
		The first estimate in Lemma \ref{maxfunc} holds more general by the same argument in the following sense. Let $u_0, u $ as above in Lemma \ref{maxfunc} and further $ 1 \leq p , q \leq \infty $ such that $ q > 4 $ and 
		\begin{align}\label{condi}
			\begin{cases}
				\f{4q}{q-4} < dp, &  q < \infty\\[3pt]
				~~~~4 < dp, &  q = \infty.
			\end{cases}
		\end{align}
		Then there holds
		\begin{align}
			\sup_{e \in \mathcal{M}}\norm{e^{\pm i t \Delta}u_0}{L^p_e L^{q}_{t, e^{\perp}}} &\leq C_{d,p,q} \lambda^{\f{d}{2}- \f{d+1}{q} - \f{1}{p}  } \norm{u_0}{L^2_{x}},
		\end{align}
		Provided \eqref{condi} holds, it is verified that 
		$$\int_{\infty}^{\infty} \left( \int_0^{\infty} (1 + |x_e| + r)^{- \f{dq}{4}} r^{d-1}~dr \right)^{\f{p}{q}} d x_e < \infty, $$
		which is required by the argument in the proof of Lemma \ref{maxfunc}, if we use
		\begin{align*}
			&\left| \int_{\R^d} e^{i x \cdot \xi} e^{\pm i t |\xi|^2} \varphi(|\xi|) d \xi \right|  \lesssim (1 + | x_e| + |(t, x_{e^{\perp}})|)^{- \f{d}{2}}.
		\end{align*}
		Under the assumption \eqref{condi}, we especially infer
		$$ \f{2}{p} + \f{d}{q} < \f{2d(q-4) + 4d}{4q} = \f{d(q-2)}{2 q} < \f{d}{2},$$
		so that  $(p,q)$ is admissible. This is a natural requirement, since typically Strichartz bounds with bounded frequency rely on estimating the truncated dispersion factor via Young's inequality.
	\end{Rem}
	\begin{Rem}
		We apply the estimates to Lemma \ref{linear-bounds} and Lemma \ref{linear-bounds-inhomogeneous} in Section \ref{sec:Linear-est-func}. Also, in Section \ref{sec:Linear-est-func}, we need to use Corollary \ref{Corol-Strichartz} and Lemma \ref{maxfunc} for functions on $\R^d$ that have Fourier support in $ A^d_{\lambda /2} \cup A^d_{\lambda } \cup A^d_{2\lambda }$. This is observed (for all $t \in \R$) e.g. for functions on $\R^{d+1}$ localized (in $(\tau, \xi)$) in $B_e \cap A_{\lambda}$, which have Fourier support in $A^d_{\lambda} \cup A^d_{\lambda/2}$, and this poses no problem to the proof.
	\end{Rem}
	\begin{center}
		\textbf{Acknowledgments}
	\end{center}
	The author acknowledges funding by the Deutsche Forschungsgemeinschaft (DFG, German Research Foundation) - Project-ID 258734477 - SFB 1173

	\bibliographystyle{alpha}

\end{document}